\documentclass[twoside]{article}

\usepackage{reductive}

\author{Harm Derksen\footnote{The research of the
first author was supported by NSF grant DMS 0349019.}{\ } and Gregor Kemper}

\title{Computing Invariants of Algebraic Group Actions in Arbitrary Characteristic}

\date{April 11, 2007}

\begin{document}

\maketitle

\begin{abstract}
  Let $G$ be an affine algebraic group acting on an affine variety
  $X$. We present an algorithm for computing generators of the
  invariant ring $K[X]^G$ in the case where $G$ is reductive.
  Furthermore, we address the case where $G$ is connected and
  unipotent, so the invariant ring need not be finitely generated. For
  this case, we develop an algorithm which computes $K[X]^G$ in terms
  of a so-called colon-operation. From this, generators of $K[X]^G$
  can be obtained in finite time if it is finitely generated. Under
  the additional hypothesis that $K[X]$ is factorial, we present an
  algorithm that finds a quasi-affine variety whose coordinate ring is
  $K[X]^G$.  Along the way, we develop some techniques for dealing
  with non-finitely generated algebras. In particular, we introduce
  the finite generation locus ideal.
\end{abstract}


\section*{Introduction} \label{0sIntro}

Throughout this article, $G$ will be an affine algebraic group over an
algebraically closed field $K$. By a $G$-variety we understand an
affine variety $X$ over $K$ with a $G$-action given by a morphism $G
\times X \to X$. The ring of regular functions on $X$ is denoted by
$K[X]$. $G$ acts on $K[X]$ by
\[
\sigma(f) = f \circ \sigma^{-1}
\]
for $\sigma \in G$ and $f \in K[X]$. The invariant ring is
\[
K[X]^G := \{f \in K[X] \mid \sigma(f) = f \ \text{for all} \ \sigma
\in G\}.
\]
\mycite{nag:d} showed that $K[X]^G$ is finitely generated as a
$K$-algebra if $G$ is reductive, i.e., the trivial group is the only
connected, unipotent, normal subgroup of $G$. On the other hand,
\mycite{pop} showed that if $G$ is not reductive, then there exists a
$G$-variety $X$ such that $K[X]^G$ is not finitely generated.
Moreover, \mycite{nag:b} showed that if $X$ is normal, then $K[X]^G$
is always isomorphic to the coordinate ring $K[U]$ of a quasi-affine
variety $U$ over $K$, even if $K[X]^G$ is not finitely generated.
Several problems arise from these facts:
\begin{enumerate}
  \renewcommand{\theenumi}{\arabic{enumi}}
\item Find an algorithm that constructs generators of $K[X]^G$ for $G$
  reductive.
\item Find an algorithm that decides whether $K[X]^G$ is finitely
  generated for $G$ non-reductive.
\item Find an algorithm that constructs generators of $K[X]^G$ if it
  is finitely generated.
\item Find an algorithm that constructs a quasi-affine variety $U$
  with $K[X]^G \cong K[U]$ (in the case that $X$ is normal).
\end{enumerate}

In the case that $K$ has characteristic~0, a solution for the first
problem was given by the first author~[\citenumber{Derksen:99}]. (More
precisely, the article [\citenumber{Derksen:99}] deals with the case
that $G$ is linearly reductive.)  The second author gave a solution of
the first problem in the case that $X = \mathbb{A}^n(K)$ is affine
$n$-space and the action of $G$ is
linear~[\citenumber{kem.separating}]. The third problem was solved by
\mycite{essen} for $G = \Ga$ being the additive group and $K$ being of
characteristic~0 (see \sref{3lMu} of this paper). Van den Essen's
algorithm terminates after finitely many steps if and only if
$K[X]^{\Ga}$ is finitely generated.

In the first and last section of this paper, we do the following:
\begin{itemize}
\item We give a complete solution to the first problem
  (\aref{2aInvarAff}). An optimized algorithm is given for the case
  that $X$ is normal and $G$ is connected (\aref{3aNormal}).
\item We give a new algorithm for computing $K[X]^G$ in the case that
  $G = \Ga$ is the additive group and $X$ is irreducible (see
  \sref{3sGap}). This algorithm works in arbitrary characteristic. As
  Van den Essen's algorithm, our algorithm first finds an $f \in
  K[X]^{\Ga} \setminus \{0\}$ and finitely many generators of the
  localization $K[X]^{\Ga}_f$. This is used for computing generators
  of $K[X]^{\Ga}$ in a second step. If the invariant ring is not
  finitely generated, this second step continues to produce generating
  invariants forever.
\item We extend the algorithm for additive group invariants to the
  case where $G$ is connected and unipotent, and $X$ is irreducible
  (\aref{7aunipotent}). The algorithm has the same nature as the one
  for the additive group. Thus we get a solution of the third problem
  for this case.
\item We find an algorithm for constructing a quasi-affine variety $U$
  with $K[X]^G \linebreak \cong K[U]$ in the case that $G$ is
  connected and unipotent, and $K[X]$ is factorial
  (\aref{7afactorial}). The isomorphism is given explicitly. This
  algorithm always terminates after finitely many steps. Thus we solve
  the fourth problem for this case.
\item We develop some ideas how the third problem can be attacked in
  general (\sref{3sAlg}).
\end{itemize}
We leave it to others to make any progress on the second problem. The
middle section of this paper deals with non-finitely generated
algebras. In the context of this paper, this prepares the ground for
the last section, but we believe that the following results from the
middle section are of more general interest:
\begin{itemize}
\item We introduce ``colon-operations'' $(R:{\mathfrak a})_S$ and
  $(R:{\mathfrak a}^\infty)_S$ and give algorithms for computing them
  in the case that $R \subseteq S$ are finitely generated algebras
  over a field and $\mathfrak a$ is an ideal of $R$ (see
  \sref{2sColon}). The coordinate ring of an irreducible, quasi-affine
  variety appears as a special case (see \lref{2lQuasiAffine}).
\item We prove that for a subalgebra $R$ of a finitely generated
  domain over a field, there always exists $f \in R \setminus \{0\}$
  such that $R_f$ is finitely generated (\pref{6ploc}). We also prove
  that the set of all these $f$'s, together with~0, forms an ideal,
  the \df{finite generation locus ideal}.
\item We give a constructive version of Grothendieck's generic
  freeness lemma (see \tref{propgenfree} and \aref{5agenfree}).
\item We give an algorithm for computing the intersection of a
  finitely generated domain over a field and the field of fractions of
  a subalgebra (\aref{5aHilb14}). This algorithm addresses the
  original version of Hilbert's fourteenth problem. Our algorithm
  terminates after finitely many steps if and only if the intersection
  is finitely generated.
\end{itemize}

\head{Acknowledgments} This work was initiated during a visit of the
second author to the University of Michigan. The second author thanks
the first one for his hospitality. Both authors thank Tobias Kamke for
carefully reading the manuscript and pointing out some errors to us.
We also thank Frank Grosshans for sending us his nice
paper~[\citenumber{Grosshans2007}] and thereby bringing a result of
\mycite{Kallen93} to our attention.

\section{Invariants of reductive groups} \label{1sReductive}

In this section we give algorithms for computing invariant rings of
reductive groups acting on affine varieties. The assumption on
reductivity of $G$ is not needed in \sref{1sEmbed}.

\subsection{Embedding into a linear space} \label{1sEmbed}

If $X = \mathbb{A}^n(K)$ is affine $n$-space and the action is linear,
we say that $X$ is a $G$-module. We usually use letters like $V$ or
$W$ for $G$-modules. A $G$-module is given by a morphism $G \to
\GL_n(K)$ of algebraic groups.

Our first goal is to embed an arbitrary $G$-variety $X$ equivariantly
into a $G$-module $V$. The idea for this is simple and standard. Since
the $G$-action on $K[X]$ is locally finite, there exists a
finite-dimensional $G$-stable vector space $W \subseteq K[X]$ which
generates $K[X]$ as a $K$-algebra. So we obtain a $G$-equivariant
epimorphism from the symmetric algebra $S(W)$ onto $K[X]$. Since $S(W)
= K[W^*]$, $V = W^*$ (the dual of $W$) is the desired $G$-module.
However, for turning this rough idea into an algorithm, we have to
work out quite a few details.

Before we can even start to formulate algorithms, we need to specify
the form of the input data.

\begin{con} \label{1cInput}
  We assume that $G$ and $X$ are given by the following data:
  \begin{enumerate}
  \item generators of a radical ideal $J \subset K[t_1 \upto t_m]$ in
    a polynomial ring such that $J$ defines $G$ as an affine variety
    in $K^m$;
  \item generators of a radical ideal $I \subseteq K[x_1 \upto x_n]$
    in another polynomial ring such that $I$ defines $X$ as an affine
    variety in $K^n$;
  \item polynomials $g_1 \upto g_n \in K[t_1 \upto t_m,x_1 \upto x_n]$
    such that for a point $(\xi_1 \upto \xi_n) \in X$ and a group
    element $\sigma = (\gamma_1 \upto \gamma_m) \in G$ we have
    \[
    \sigma(\xi_1 \upto \xi_n) =
    \left(g_1(\underline{\gamma},\underline{\xi}) \upto
      g_n(\underline{\gamma},\underline{\xi})\right),
    \]
    where we write $(\underline{\gamma})$ for $(\gamma_1 \upto
    \gamma_m)$ etc.
  \end{enumerate}
\end{con}

We are now ready to formulate our first algorithm.

\begin{alg}[Embedding $X$ into a $G$-module $V$] \label{1aEmbed}
  \mbox{}
  \begin{description}
  \item[Input:] An affine algebraic group $G$ and a $G$-variety $X$ given
    according to \conref{1cInput}.
  \item[Output:] Polynomials $a_{i,j} \in K[t_1 \upto t_n]$ ($i,j \in
    \{1 \upto r\}$) such that
    \[
    G \to \GL_r(K), \ (\gamma_1 \upto \gamma_m) \mapsto
    \begin{pmatrix}
      a_{1,1}(\underline{\gamma}) & \cdots &
      a_{1,r}(\underline{\gamma}) \\
      \vdots & & \vdots \\
      a_{r,1}(\underline{\gamma}) & \cdots &
      a_{r,r}(\underline{\gamma})
    \end{pmatrix}
    \]
    defines a $G$-module $V = K^r$, and polynomials $h_1 \upto h_r \in
    K[x_1 \upto x_n]$ such that
    \[
    X \to V, \ (\xi_1 \upto \xi_n) \mapsto \left(h_1(\underline{\xi})
      \upto h_r(\underline{\xi})\right)
    \]
    is $G$-equivariant and injective.
  \end{description}
  \begin{enumerate}
    \renewcommand{\theenumi}{\arabic{enumi}}
  \item \label{s11} Compute Gr\"obner bases $\mathcal{G}_I$ and
    $\mathcal{G}_J$ of $I$ and $J$ with respect to arbitrary monomial
    orderings on $K[x_1 \upto x_n]$ and $K[t_1 \upto t_m]$.
  \item \label{s12} Substitute each $g_i$ by its normal form
    $\NF_{\mathcal{G}_I \cup {\mathcal G}_J}(g_i)$. (This means that
    whenever a monomial of $g_i$ is divisible by a leading monomial of
    an element of $\mathcal{G}_I$ or $\mathcal{G}_J$, the
    corresponding reduction should be performed.)
  \item \label{s13} Let $C \subseteq K[x_1 \upto x_n]$ be the set of
    all coefficients occurring in the $g_i$ considered as polynomials
    in $t_1 \upto t_m$.
  \item \label{s14} Select a maximal $K$-linearly independent subset
    $\{h_1 \upto h_r\} \subseteq C$.
  \item \label{s15} For $i = 1 \upto r$, form
    \[
    \tilde{h_i} := \NF_{\mathcal{G}_I \cup {\mathcal
    G}_J}\left(h_i(g_1 \upto g_n)\right) \in K[t_1 \upto t_m,x_1 \upto
    x_n].
    \]
  \item \label{s16} For $i = 1 \upto r$, find $a_{i,1} \upto a_{i,r}
    \in K[t_1 \upto t_m]$ such that
    \begin{equation} \label{1eqAij}
      \tilde{h_i} = \sum_{j=1}^r a_{i,j} h_j.
    \end{equation}
    This can be done by viewing~\eqref{1eqAij} as an equation in
    \linebreak $K(t_1 \upto t_m)[x_1 \upto x_n]$, comparing
    coefficients in the $x$-variables, and solving the resulting
    linear system with coefficients in $K(t_1 \upto t_m)$. In fact,
    there exists a unique solution, which lies in $K[t_1 \upto
    t_m]^r$.
  \end{enumerate}
\end{alg}

\begin{proof}[Proof of correctness of \aref{1aEmbed}]
  We first remark that converting the $g_i$ into normal form
  (Step~\ref{s12}) does not change their properties given in
  \conref{1cInput}(c). We will assume that $g_i$ are in normal form.

  Throughout the proof let $\sigma = (\gamma_1 \upto \gamma_m)$ and
  $\tau = (\eta_1 \upto \eta_m)$ be elements from $G$, and write
  $\sigma \tau = (\zeta_1 \upto \zeta_m)$ for their product. For
  $(\xi_1 \upto \xi_n) \in X$ we have
  \[
  \left(\sigma^{-1}(x_i + I)\right)(\xi_1 \upto \xi_n) = (x_i +
  I)\left(\sigma(\xi_1 \upto \xi_n)\right) =
  g_i(\underline{\gamma},\underline{\xi}),
  \]
  so
  \begin{equation} \label{1eqSigma}
    \sigma^{-1}(x_i + I) = g_i(\underline{\gamma},\underline{x}) + I.
  \end{equation}
  We can write
  \[
  g_i = \sum_{j=1}^l h_{i,j} f_j
  \]
  with $f_1 \upto f_l \in K[t_1 \upto t_m]$ pairwise distinct
  monomials in normal form w.r.t. $\mathcal{G}_J$ and $h_{i,j} \in
  K[x_1 \upto x_n]$ in normal form w.r.t. $\mathcal{G}_I$. With this,
  \eqref{1eqSigma} becomes
  \begin{equation} \label{1eqFH}
    \sum_{j=1}^l f_j(\underline{\gamma}) (h_{i,j} + I) =
    \sigma^{-1}(x_i + I).
  \end{equation}
  Let
  \[
  W := \sum_{i=1}^n \sum_{j=1}^l K \cdot (h_{i,j} + I)
  \]
  be the subspace of $K[X]$ generated by the residue classes of all
  $h_{i,j}$. With the $h_i$ selected as in Step~\ref{s14}, a $K$-basis
  of $W$ is given by $h_1 + I \upto h_r + I$. From~\eqref{1eqFH} with
  $\sigma$ being the identity element, we see that $x_i + I \in W$ for
  all~$i$, so $K[X]$ is generated by $h_1 + I \upto h_r + I$ as a
  $K$-algebra. This implies that the map $X \to K^r = V$ given by the
  $h_i$ is injective.
  
  Applying $\tau^{-1}$ to~\eqref{1eqFH} and then
  applying~\eqref{1eqFH} with $\sigma \tau$ in the place of $\sigma$
  yields
  \[
  \sum_{j=1}^l f_j(\underline{\gamma}) \cdot \tau^{-1}(h_{i,j} + I) =
  \tau^{-1}\left(\sigma^{-1}(x_i + I)\right) = \sum_{j=1}^l
  f_j(\underline{\zeta}) (h_{i,j} + I) \in W.
  \]
  Since the $f_j$ are linearly independent as functions on $G$, this
  shows that all $\tau^{-1}(h_{i,j} + I)$ lie in $W$, so $W$ is
  $G$-stable. To see that the $a_{i,j}$ from Step~\ref{s16} exist,
  choose a set $B \subseteq K[x_1 \upto x_n]$ such that the $h + I$
  with $h \in B$ together with all $h_i + I$ form a $K$-basis of
  $K[X]$. Then for $i \in \{1 \upto n\}$ we can write
  \[
  \tilde{h_i} + I = \sum_{j=1}^r a_{i,j} h_j + \sum_{j=1}^s
  a_{i,j}' h_j' + I
  \]
  with $h_j' \in B$ and $a_{i,j},a_{i,j}' \in K[t_1 \upto t_m]$. As
  $\tilde{h_i}$ is in reduced form w.r.t. $\mathcal{G}_J$, the same
  holds for all $a_{i,j}$ and $a_{i,j}'$. The definition of
  $\tilde{h_i}$, Equation~\eqref{1eqSigma} and the $G$-stability of
  $W$ imply
  \[
  \tilde{h_i}(\gamma_1 \upto \gamma_m,x_1 \upto x_n) + I =
  \sigma^{-1}(h_i + I) \in W,
  \]
  so
  \[
  \sum_{j=1}^r a_{i,j}(\underline{\gamma}) h_j + \sum_{j=1}^s
  a_{i,j}'(\underline{\gamma}) h_j' + I \in W
  \]
  for all $\sigma = (\gamma_1 \upto \gamma_m) \in G$. Since $W$ is
  generated by the $h_i + I$, it follows that all
  $a_{i,j}'(\underline{\gamma})$ are zero, so $a_{i,j}' \in J$. Since
  they are in normal form, $a_{i,j}' = 0$ for all~$j$, so $\tilde{h_i}
  + I = \sum_{j=1}^r a_{i,j} h_j + I$. Since all polynomials in this
  equation are in reduced form w.r.t. $\mathcal{G}_I$, it follows that
  this is an equality in $K[t_1 \upto t_m,x_1 \upto x_n]$. So the
  $a_{i,j}$ from Step~\ref{s16} indeed exist. Their uniqueness follows
  from the fact that $h_1 \upto h_r$ are linearly independent over
  $K$, thus also over the rational function field $K(t_1 \upto t_m)$.
  
  Next we show that the $a_{i,j}$ define a $G$-module $V =
  K^r$. Indeed, we have
  \begin{multline*}
    \sum_{j=1}^r a_{i,j}(\underline{\zeta}) (h_j + I)= (\sigma
    \tau)^{-1} (h_i +I) = \\
    \tau^{-1}\left(\sum_{k=1}^r a_{i,k}(\underline{\gamma}) (h_k +
      I)\right) = \sum_{k=1}^r a_{i,k}(\underline{\gamma})
    \sum_{j=1}^r a_{k,j}(\underline{\eta}) (h_j + I),
  \end{multline*}
  so $a_{i,j}(\underline{\zeta}) = \sum_{k=1}^r
  a_{i,k}(\underline{\gamma}) a_{k,j}(\underline{\eta})$ by the linear
  independence of the $h_j + I$. Finally, the map $\map{\Phi}{X}{V}$
  given in \aref{1aEmbed} is $G$-equivariant, since for all $(\xi_1
  \upto \xi_n) \in X$ we have
  \begin{multline*}
    \Phi\left(\sigma(\underline{\xi})\right) =
    \left(\tilde{h_1}(\underline{\gamma},\underline{\xi}) \upto
      \tilde{h_r}(\underline{\gamma},\underline{\xi})\right) = \\
    \begin{pmatrix}
      a_{1,1}(\underline{\gamma}) & \cdots &
      a_{1,r}(\underline{\gamma}) \\
      \vdots & & \vdots \\
      a_{r,1}(\underline{\gamma}) & \cdots &
      a_{r,r}(\underline{\gamma})
    \end{pmatrix} \cdot \Phi(\underline{\xi}) =
    \sigma\left(\Phi(\underline{\xi})\right).
  \end{multline*}
  This completes the proof.
\end{proof}

\subsection{Inseparable closure} \label{2sInsep}

For $R$ an algebra over a field $K$ of characteristic $p > 0$ and $A
\subseteq R$ a subalgebra, we write
\[
\sqrt[p]{A} := \left\{g \in R \mid g^p \in A\right\}
\]
and call this the $p$-th root of $A$ in $R$. Moreover,
\[
\widehat{A} := \left\{g \in R \mid g^q \in A \ \text{for some
  $p$-power $q$}\right\}
\]
is called the \df{inseparable closure} of $A$ in $R$. $\sqrt[p]{A}$
and $\widehat{A}$ are clearly $A$-modules and $K$-algebras. The
following remark sheds some light on the importance of the inseparable
closure to invariant theory.

\begin{rem} \label{2rSeparating}
  Suppose that $G$ is a reductive group over an algebraically closed
  field $K$ of positive characteristic, and $V$ is a $G$-module. Let
  $A \subseteq K[V]^G$ be a \df{separating} subalgebra. By definition,
  this means that $A$ has the same capabilities of separating
  $G$-orbits as $K[V]^G$ (see
  \mycite[Definition~2.3.8]{Derksen:Kemper}). Since the natural map $V
  \to \Specm\left(K[V]^G\right)$ is surjective, this implies that the
  map $\Specm\left(K[V]^G\right) \to \Specm\left(A\right)$ is
  injective. Assume further that $A$ is generated by homogeneous
  invariants. Then Theorem~2.3.12 of~[\citenumber{Derksen:Kemper}]
  implies that $K[V]^G$ is integral over $A$. By
  \mycite[Sublemma~A.5.1]{Kallen93} (for an expanded version of the
  proof see http://www.math.uu.nl/people/vdkallen/\linebreak
  errbmod.pdf), the integrality and the injectiveness of the
  corresponding morphism imply that $K[V]^G \subseteq \widehat{A}$.
  Here the inseparable closure can and will be understood to be formed
  in $K[V]$. Since $\widehat{K[V]^G} = K[V]^G$ is always true, we
  conclude
  \begin{equation} \label{2eqKallen}
    \widehat{A} = K[V]^G.
  \end{equation}
  (In fact, the converse is also true: If a subalgebra $A \subseteq
  K[V]^G$ satisfies~\eqref{2eqKallen}, then it is separating.) The
  conclusion~\eqref{2eqKallen} is an improvement
  of~[\citenumber{Derksen:Kemper}, Theorem~2.3.12], which says that
  $K[V]^G$ is obtained from $A$ by first taking the normalization and
  then the inseparable closure. This improvement only holds in
  positive characteristic. Using~\eqref{2eqKallen}, we also get an
  improvement to the algorithm given by \mycite{kem.separating} for
  computing $K[V]^G$. In fact, Algorithm~1.9
  of~[\citenumber{kem.separating}] first calculates the normalization
  (Step~2) and then the inseparable closure (Step~3). Thus in positive
  characteristic, Step~2 can in fact be omitted.
\end{rem}

In \mycite[Algorithm~4.2]{kem.separating} an algorithm is given for
computing $\sqrt[p]{A}$ in the case that $R$ is a polynomial ring. We
need to modify this algorithm substantially to make it suitable for
the case that $R$ is any reduced finitely generated $K$-algebra.

\begin{alg}[$p$-th root of a subalgebra] \label{2aPRoot} \mbox{}
  \begin{description}
  \item[Input:] Polynomials $h_1 \upto h_l \in K[x_1 \upto x_n]$ over
    a perfect field $K$ of characteristic $p > 0$ such that $I = (h_1
    \upto h_l)$ is a radical ideal, and polynomials $f_1 \upto f_m \in
    K[x_1 \upto x_n]$ defining a subalgebra $A := K[f_1+I \upto f_m+I]
    \subseteq R := K[x_1 \upto x_n]/I$.
  \item[Output:] Polynomials $g_1 \upto g_r \in K[x_1 \upto x_n]$ such
    that
    \[
    \sqrt[p]{A} = \sum_{i=1}^r A \cdot (g_i + I).
    \]
  \end{description}
  \begin{enumerate}
    \renewcommand{\theenumi}{\arabic{enumi}}
  \item Let $F$ be a free $K[x_1 \upto x_n]$-module of rank $(p^m + l
    p^n + 1)$ with basis vectors $e_{i_1 \upto i_m}$ ($i_\nu \in \{0
    \upto p-1\}$), $e_{i_1 \upto i_n}^{(j)}$ ($j \in \{1 \upto l\}$,
    $i_\nu \in \{0 \upto p-1\}$), and $e^{(0)}$.
  \item Form the $K[x_1 \upto x_n]$-module $M \subseteq F$ formed by
    all
    \[
    e_{i_1 \upto i_m} + \prod_{\nu=1}^m f_\nu^{i_\nu} e^{(0)} \quad
    (i_\nu \in \{0 \upto p-1\})
    \]
    and
    \[
    e_{i_1 \upto i_n}^{(j)} + \prod_{\nu=1}^n x_\nu^{i_\nu} h_j
    e^{(0)} \quad (j \in \{1 \upto l\}, \ i_\nu \in \{0 \upto p-1\}).
    \]
  \item Let $K[y_1 \upto y_n]$ be a new polynomial ring and write
    $\phi$ for the map $K[y_1 \upto y_n] \to K[x_1 \upto x_n]$ sending
    each $y_i$ to $x_i^p$. Also use the letter $\phi$ for the
    component-wise application of $\phi$ to the free module \linebreak
    $K[y_1 \upto y_n]^{p^m + l p^m + 1}$.
  \item \label{s24} Use \aref{2aIntersect} below to compute $C_1 \upto
    C_s \in K[y_1 \upto y_n]^{p^m + l p^m + 1}$ such that the
    $\phi(C_i)$ generate
    \[
    M \cap K[x_1^p \upto x_n^p]^{p^m + l p^m + 1}
    \]
    as a $K[x_1^p \upto x_n^p]$-module.
  \item \label{s25} With $\map{\pi}{K[y_1 \upto y_n]^{p^m + l p^m +
    1}}{K[y_1 \upto y_n]^{p^m}}$ the projection on the first $p^m$
    coordinates, form
    \[
    \tilde{M} := \sum_{i=1}^s K[y_1 \upto y_n] \cdot \pi(C_i)
    \subseteq K[y_1 \upto y_n]^{p^m}.
    \]
    Moreover, form $\tilde{f}_1 \upto \tilde{f}_m \in K[y_1 \upto
    y_n]$ from the $f_i$ by raising each coefficient of $f_i$ to its
    $p$-th power and substituting each $x_j$ by $y_j$.
  \item \label{s26} Use \aref{2aIntersect} to compute generators $s_1
    \upto s_r$ of $\tilde{M} \cap K[\tilde{f}_1 \upto
    \tilde{f}_m]^{p^m}$ as a module over $K[\tilde{f}_1 \upto
    \tilde{f}_m]$ and a matrix $(a_{i,j}) \in K[y_1 \upto y_n]^{r
    \times s}$ such that
    \[
    s_i = \sum_{j=1}^s a_{i,j} \pi(C_j).
    \]
  \item \label{s27} For $i = 1 \upto r$, let $g_i \in K[x_1 \upto
    x_n]$ be the (unique) $p$-th root of
    \[
    \sum_{j=1}^s \phi(a_{i,j}) \cdot \phi\left(C^{(0)}_j\right) \in
    K[x_1^p \upto x_n^p],
    \]
    where $C^{(0)}_j$ is the $e^{(0)}$-component of $C_j$.
  \end{enumerate}
\end{alg}

\begin{proof}[Proof of correctness of \aref{2aPRoot}]
  Throughout the proof we write $\overline{g} := g + I \in R$ for the
  residue class of a polynomial $g \in K[x_1 \upto x_n]$. Take an
  element
  \[
  (\underline{u}) = \sum_{i_1 \upto i_m=0}^{p-1} u_{i_1 \upto i_m}
  e_{i_1 \upto i_m} + \sum_{j=1}^l \sum_{i_1 \upto i_n = 1}^{p-1}
  u_{i_1 \upto i_n}^{(l)} e_{i_1 \upto i_n}^{(l)} + u^{(0)} e^{(0)}
  \]
  from $F$ (with all $u$'s from $K[x_1 \upto x_n]$). Then
  $(\underline{u}) \in M$ implies
  \[
  \sum_{i_1 \upto i_m=0}^{p-1} u_{i_1 \upto i_m} \cdot \prod_{\nu=1}^m
  f_\nu^{i_\nu} + \sum_{j=1}^l \sum_{i_1 \upto i_n = 1}^{p-1} u_{i_1
  \upto i_n}^{(l)} \cdot \prod_{\nu=1}^n x_\nu^{i_\nu} h_j - u^{(0)} =
  0,
  \]
  so
  \begin{equation} \label{2eqM}
    \overline{u^{(0)}} = \sum_{i_1 \upto i_m=0}^{p-1} \overline{u_{i_1
      \upto i_m}} \cdot \prod_{\nu=1}^m \overline{f_\nu}^{i_\nu}.
  \end{equation}
  
  First we show that all $\overline{g_i}^p$ lie in $A$. All
  $\phi(C_j)$ lie in $M$, and therefore also $\sum_{j=1}^s
  \phi(a_{i,j}) \phi(C_j) \in M$. The $e^{(0)}$-component of
  $\sum_{j=1}^s \phi(a_{i,j}) \phi(C_j)$ is $g_i^p$ by Step~\ref{s27}
  of the algorithm. Moreover, for all $i_1 \upto i_m \in \{0 \upto
  p-1\}$, the $e_{i_1 \upto i_m}$-component of $\sum_{j=1}^s a_{i,j}
  C_j$ is equal to the corresponding component of $s_i$ by
  Step~\ref{s26}, and $s_i$ lies in $K[\tilde{f}_1 \upto
  \tilde{f}_m]$. Thus the $e_{i_1 \upto i_m}$-component of
  $\sum_{j=1}^s \phi(a_{i,j}) \phi(C_j)$ lies in $K[\phi(\tilde{f}_1)
  \upto \phi(\tilde{f}_m)]$. But $\phi(\tilde{f}_j) = f_j^p$ by the
  definition of the $\tilde{f}_j$, so from~\eqref{2eqM} we obtain
  \[
  \overline{g_i}^p = \sum_{i_1 \upto i_m=0}^{p-1} \overline{u_{i_1
    \upto i_m}} \cdot \prod_{\nu=1}^m \overline{f_\nu}^{i_\nu}
  \]
  with $u_{i_1 \upto i_m}$ elements from $K[f_1^p \upto f_m^p]$. Hence
  indeed $\overline{g_i}^p \in A$.
  
  Now we show that every element from $\sqrt[p]{A}$ is an $A$-linear
  combination of $\overline{g_1} \upto \overline{g_r}$. So take $g \in
  K[x_1 \upto x_n]$ such that $\overline{g} \in \sqrt[p]{A}$. This
  means that $\overline{g}^p \in A \cap K[\overline{x_1}^p \upto
  \overline{x_m}^p]$. So on the one hand there exists $u^{(0)} \in
  K[x_1^p \upto x_n^p]$ with $\overline{g}^p = \overline{u^{(0)}}$,
  and on the other hand we have $u_{i_1 \upto i_m} \in K[f_1^p \upto
  f_m^p]$ (for $i_1 \upto i_m \in \{0 \upto p-1\}$) such that
  \begin{equation} \label{2eqGp}
  \overline{u^{(0)}} = \overline{g}^p = \sum_{i_1 \upto i_m = 0}^{p-1}
    \overline{u_{i_1 \upto i_m}} \cdot \prod_{\nu=1}^m
    \overline{f_\nu}^{i_\nu}.
  \end{equation}
  Indeed, any element of $A$ can be written like this. But this means
  that there exist polynomials $u_{i_1 \upto i_n}^{(j)} \in K[x_1^p
  \upto x_n^p]$ (for $j \in \{1 \upto l\}$ and $i_1 \upto i_m \in \{0
  \upto p-1\}$) such that
  \begin{equation} \label{2eqI}
    u^{(0)} - \sum_{i_1 \upto i_m = 0}^{p-1} u_{i_1 \upto i_m} \cdot
    \prod_{\nu=1}^m f_\nu^{i_\nu} = \sum_{j=1}^l \sum_{i_1 \upto i_n =
    1}^{p-1} u_{i_1 \upto i_n}^{(j)} \cdot \prod_{\nu=1}^n
    x_\nu^{i_\nu} h_j.
  \end{equation}
  Indeed, any element from $I$ can be written as an expression as on
  the right hand side of~\eqref{2eqI}. Equation~\eqref{2eqI} implies
  that the element
  \[
  (\underline{u}) = \sum_{i_1 \upto i_m=0}^{p-1} u_{i_1 \upto i_m}
  e_{i_1 \upto i_m} + \sum_{j=1}^l \sum_{i_1 \upto i_n = 1}^{p-1}
  u_{i_1 \upto i_n}^{(l)} e_{i_1 \upto i_n}^{(l)} + u^{(0)} e^{(0)}
  \]
  of $F$ lies in $M$. Observe that all coefficients of
  $(\underline{u})$ lie in $K[x_1^p \upto x_n^p]$. Thus by
  Step~\ref{s24} of the algorithm, $(\underline{u})$ lies in the
  $K[x_1^p \upto x_n^p]$-span of the $\phi(C_i)$. It is convenient to
  write $u_{i_1 \upto i_m} = \phi(U_{i_1 \upto i_m})$ with $U_{i_1
  \upto i_m} \in K[y_1 \upto y_n]$. Then
  \[
  \sum_{i_1 \upto i_m = 0}^{p-1} U_{i_1 \upto i_m} \cdot e_{i_1 \upto
  i_m} \in \tilde{M}
  \]
  with $\tilde{M}$ as defined in Step~\ref{s25}. But we know that the
  $u_{i_1 \upto i_m}$ really lie in $K[f_1^p \upto f_m^p]$, which
  implies $U_{i_1 \upto i_m} \in K[\tilde{f}_1 \upto \tilde{f}_m]$. So
  by Step~\ref{s26} there exist $B_1 \upto B_r \in K[\tilde{f}_1 \upto
  \tilde{f}_m]$ such that
  \[
  \sum_{i_1 \upto i_m = 0}^{p-1} U_{i_1 \upto i_m} \cdot e_{i_1 \upto
  i_m} = \sum_{i=1}^r B_i s_i = \sum_{i=1}^r B_i \cdot \sum_{j=1}^s
  a_{i,j} \pi(C_j).
  \]
  Applying $\phi$ to this and setting $b_i := \phi(B_i) \in K[f_1^p
  \upto f_m^p]$ yields
  \[
  \sum_{i_1 \upto i_m = 0}^{p-1} u_{i_1 \upto i_m} \cdot e_{i_1 \upto
  i_m} = \sum_{i=1}^r b_i \sum_{j=1}^s \phi(a_{i,j}) \cdot \sum_{i_1
  \upto i_m = 0}^{p-1} \phi(C_j^{(i_1 \upto i_m)}) e_{i_1 \upto i_m},
  \]
  where $C_j^{(i_1 \upto i_m)}$ stands for the $e_{i_1 \upto
    i_m}$-component of $C_j$. So for every $i_1 \upto i_m \linebreak
  \in \{0 \upto p-1\}$ we have
  \[
  u_{i_1 \upto i_m} = \sum_{i=1}^r b_i \sum_{j=1}^s \phi(a_{i,j})
  \cdot \phi(C_j^{(i_1 \upto i_m)}).
  \]
  Substituting this into~\eqref{2eqGp} yields
  \[
  \overline{g}^p = \sum_{i=1}^r \overline{b_i} \sum_{j=1}^s
  \overline{\phi(a_{i,j})} \cdot \sum_{i_1 \upto i_m = 0}^{p-1}
  \overline{\phi(C_j^{(i_1 \upto i_m)})} \cdot \prod_{\nu=1}^m
  \overline{f_\nu}^{i_\nu}.
  \]
  But $\phi(C_j) \in M$ for all~$j$, so we can apply~\eqref{2eqM} and
  obtain
  \[
  \overline{g}^p = \sum_{i=1}^r \overline{b_i} \sum_{j=1}^s
  \overline{\phi(a_{i,j})} \cdot \overline{\phi(C_j^{(0)})} =
  \sum_{i=1}^r \overline{b_i} \overline{g_i}^p,
  \]
  where the last equality follows from Step~\ref{s27}. Since $b_i \in
  K[f_1^p \upto f_m^p]$ and since $K$ is perfect, there exist $p$-th
  roots of the $b_i$ in $K[f_1 \upto f_m]$. Hence there exist
  $\tilde{b}_i \in A$ with $\tilde{b}_i^p = \overline{b_i}$. We obtain
  \[
  \overline{g}^p = \sum_{i=1}^r \tilde{b}_i^p \overline{g_i}^p =
  \left(\sum_{i=1}^r \tilde{b}_i \overline{g_i}\right)^p.
  \]
  Since $I$ is a radical ideal, this implies $\overline{g} =
  \sum_{i=1}^r \tilde{b}_i \overline{g_i}$ with $\tilde{b}_i \in
  A$. This completes the proof.
\end{proof}

The following algorithm is used in \aref{2aPRoot}. It is a slight
extension of Algorithm~7 in \mycite{kem:e} (see also
\mycite[Section~3.6, Exercise~10~c]{Kreuzer:Robbiano}).

\begin{alg}[Intersection of a submodule with a subalgebra]
  \label{2aIntersect} \mbox{}
  \begin{description}
  \item[Input:] Generators $b_1 \upto b_l$ of a submodule $M \subseteq
    K[x_1 \upto x_n]^r$, and polynomials $f_1 \upto f_m \in K[x_1
    \upto x_n]$ generating a subalgebra $A := \linebreak K[f_1 \upto
    f_m]$.
  \item[Output:]
    \begin{enumerate}
    \item[-] Generators $c_1 \upto c_s$ of $M \cap A^r$ as an $A$-module;
    \item[-] if desired, elements $C_1 \upto C_s \in K[y_1 \upto
      y_m]^r$ with $K[y_1 \upto y_m]$ a polynomial ring such that
      substituting $y_i \mapsto f_i$ in $C_i$ yields $c_i$;
    \item[-] if desired, a matrix $(a_{i,j}) \in K[x_1 \upto x_n]^{s
      \times l}$ such that
      \begin{equation} \label{2eqCi}
        c_i = \sum_{j=1}^l a_{i,j} b_j
      \end{equation}
      for all $i \in \{1 \upto s\}$.
    \end{enumerate}
  \end{description}
  \begin{enumerate}
    \renewcommand{\theenumi}{\arabic{enumi}}
  \item Let $S := K[x_1 \upto x_n,y_1 \upto y_m]$ be a polynomial ring
    with additional indeterminates $y_1 \upto y_m$. Form the submodule
    $\widetilde{M}$ of $S^r$ generated by $b_i$ ($i = 1 \upto l$) and
    by $(f_j - y_j) \cdot e_k$ ($j = 1 \upto m$, $k = 1 \upto r$),
    where the $e_k$ are the free generators of $S^r$.
  \item Choose a monomial ordering ``$>$'' on $S^r$ such that
    \[
    x_i e_j > y_1^{d_1} \cdots y_m^{d_m} e_{j'}
    \]
    for all $i \in \{1 \upto n\}$, $j,j' \in \{1 \upto r\}$, and $d_k
    \in \NN$.
  \item \label{4nMtilde} Compute a Gr\"obner basis $\mathcal{G}$ of
    $\widetilde{M}$ with respect to ``$>$''. If the matrix $(a_{i,j})$
    is desired, keep track of how each element from $\mathcal{G}$ can
    be represented as an $S$-linear combination of the $b_i$ and $(f_j
    - y_j) \cdot e_k$.
  \item Let $C_1 \upto C_s$ be those elements from $\mathcal{G}$ which
    lie in $K[y_1 \upto y_m]^r$, and obtain $c_i$ by substituting $y_i
    \mapsto f_i$ in $C_i$.
  \item \label{s35} If the matrix $(a_{i,j})$ is desired, use the
    normal form algorithm to express each $c_i$ as an $S$-linear
    combination of the elements of $\mathcal{G}$, and then as a linear
    combination of the $b_j$ and $(f_j - y_j) \cdot e_k$:
    \begin{equation} \label{2eqCi2}
      c_i = \sum_{j=1}^l \tilde{a}_{i,j} b_j + \sum_{j=1}^m
      \sum_{k=1}^r \tilde{a}_{i,j,k} (f_j - y_j) \cdot e_k
    \end{equation}
    with $\tilde{a}_{i,j},\tilde{a}_{i,j,k} \in S$. Then $a_{i,j}$ is
    obtained by substituting $y_k \mapsto f_k$ in $\tilde{a}_{i,j}$.
  \end{enumerate}
\end{alg}

\begin{proof}[Proof of correctness of \aref{2aIntersect}]
  We only need to prove the correctness of step~\ref{s35}, since
  everything else is already contained in Algorithm~7 from \linebreak
  \mycite{kem:e}. First, the $c_i$ are contained in $M$ and therefore
  in $\widetilde{M}$, so the normal form is zero. Hence the
  $\tilde{a}_{i,j}$ and $\tilde{a}_{i,j,k}$ in~\eqref{2eqCi2} exist.
  Now substituting $y_\nu \mapsto f_\nu$ in~\eqref{2eqCi2}
  yields~\eqref{2eqCi}.
\end{proof}

\begin{remark} \label{2remmodintersect}
Algorithm~\ref{2aIntersect} can be generalized to arbitrary finitely generated
commutative $K$-algebras. Suppose that $A=K[x_1,\dots,x_l]/J$ is a subalgebra
of a $K$-algebra $B=K[x_1,\dots,x_n]/ I$ with $l \le n$ (so $I\cap K[x_1,\dots,x_l]=J$),
and $M$ is a $B$-submodule of $B^r$. Consider the quotient map
$$
\phi:K[x_1,\dots,x_n]^r\to B^r=K[x_1,\dots,x_n]^r/I^r
$$
and define $\overline{M}=\phi^{-1}(M)$. To compute $M\cap A^r$,
note that
\begin{multline*}
  M\cap A^r = M \cap \varphi(K[x_1,\dots,x_l]^r)= \\
  \varphi(\varphi^{-1}(M)\cap K[x_1,\dots,x_l]^r)=
  \varphi(\overline{M}\cap K[x_1,\dots,x_l]^r).
\end{multline*}
Generators of $\overline{M}\cap K[x_1,\dots,x_l]^r$ can be computed 
using \aref{2aIntersect}.
\end{remark}

We are now ready to present an algorithm for computing generating
invariants of a reductive groups acting on an affine variety. Recall
that every reductive group in characteristic~0 is linearly reductive,
so Derksen's algorithm~[\citenumber{Derksen:99}] applies for computing
its invariant rings. Therefore we may assume that the characteristic
is positive.

\begin{alg}[Invariants of a reductive group acting on an affine
  variety] \label{2aInvarAff} \mbox{}
  \begin{description}
  \item[Input:] A reductive algebraic group $G$ over an algebraically
    closed field $K$ of characteristic~$p$, and a $G$-variety $X$
    given according to \conref{1cInput}.
  \item[Output:] Polynomials $f_1 \upto f_k \in K[x_1 \upto x_n]$ such
    that the residue classes $f_i + I \in K[X]$ are $G$-invariant and
    generate $K[X]^G$.
  \end{description}
  \begin{enumerate}
    \renewcommand{\theenumi}{\arabic{enumi}}
  \item \label{s41} Use \aref{1aEmbed} to calculate an equivariant
    embedding $X \to V$ into a $G$-module $V$. Let $h_1 \upto h_r \in
    K[x_1 \upto x_n]$ be the polynomials by which this embedding is
    given, and write $K[V] = K[y_1 \upto y_r]$ with $y_i$
    indeterminates.
  \item \label{s42} Use Algorithm~1.9 of \mycite{kem.separating} to
    compute generators $F_1 \upto F_k \in K[y_1 \upto y_r]$ of
    $K[V]^G$. In fact, it is enough if $F_1 \upto F_k$ are
    homogeneous, separating invariants, as computed by Algorithm~2.9
    of~[\citenumber{kem.separating}], in which case $K[V]^G$ will be
    the inseparable closure of $K[F_1 \upto F_k]$ (see
    \rref{2rSeparating}).
  \item \label{s43} For $i = 1 \upto k$, set
    \[
    f_i := F_i(h_1 \upto h_r) \in K[x_1 \upto x_n],
    \]
    and let $A \subseteq K[X]$ be the $K$-algebra generated by the
    $\overline{f_i} := f_i + I \in K[X]$.
  \item \label{s44} Use \aref{2aPRoot} to compute $\sqrt[p]{A}
    \subseteq K[X]$. Let $S$ be the set of generators of $\sqrt[p]{A}$
    returned by \aref{2aPRoot}.
  \item \label{s45} For each $\overline{g} \in S$, test whether
    $\overline{g} \in A$ (see \rref{2rMember}). If $\overline{g}
    \notin A$, set
    \[
    f_{k+1} := h, \quad A:=K[\overline{f_1} \upto \overline{k_{k+1}}]
    \quad \text{and} \quad k := k+1.
    \]
  \item If in Step~\ref{s45} all $\overline{g} \in S$ were found to
    already lie in $A$, then $K[X]^G = A$ and we are done. Otherwise,
    go back to Step~\ref{s44}.
  \end{enumerate}
\end{alg}

\begin{remark} \label{2rMember}
  The membership test in Step~\ref{s45} of \aref{2aInvarAff} can be
  done as follows: With additional indeterminates $t,t_1 \upto t_k$
  choose a monomial ordering on $K[t,t_1 \upto t_k,x_1 \upto x_n]$
  such that every monomial in $t,t_1 \upto t_k$ is smaller than any
  $x_i$, and every monomial in $t_1 \upto t_k$ is smaller than~$t$.
  Compute a Gr\"obner basis $\mathcal{G}$ of the ideal in $K[t,t_1
  \upto t_k,x_1 \upto x_n]$ generated by
  \[
  g - t, \quad f_i - t_i \quad (i = 1 \upto k), \quad \text{and} \quad
  I
  \]
  with respect to this monomial ordering. Then $\overline{g} \in A$ if
  and only if $\mathcal{G}$ contains a polynomial with the lead
  monomial~$t$. This can be viewed as a (very) special case of
  \aref{2aIntersect}.
\end{remark}

\begin{proof}[Proof of correctness of \aref{2aInvarAff}]
  With $\map{\phi}{X}{V}$ the map given in Step~\ref{s41} of the
  algorithm, we have a $G$-equivariant epimorphism
  \[
  \mapl{\phi^*}{K[V]}{K[X]}{F}{F \circ \phi}
  \]
  of $K$-algebras, and $f_i + I$, as formed in Step~\ref{s43}, is just
  the $\phi^*$-image of $F_i$. Thus $A = K[f_1 + I \upto f_k + I]$,
  also formed in Step~\ref{s43}, is a subalgebra of $K[X]^G$. The
  algorithm keeps increasing~$k$ and enlarging $A$ until reaching the
  inseparable closure $\widehat{A}$. In this proof, the letter $A$
  will always denote the subalgebra formed in Step~\ref{s43}.
  
  Since $K[X]$ is a reduced ring, clearly every $\overline{g} \in
  \widehat{A}$ is an invariant in $K[X]^G$. Conversely, take
  $\overline{g} \in K[X]^G$. Since $G$ is reductive, there exists a
  $p$-power $s$ such that $\overline{g}^s \in
  \phi^*\left(K[V]^G\right)$ (see \mycite[Lemma~A.1.2]{MFK}), so
  $\overline{g}^s = \phi^*(F)$ with $F \in K[V]^G$. Since $K[V]^G$ is
  the inseparable closure of $K[F_1 \upto F_k]$, there exists a
  $p$-power $q$ with $F^q \in K[F_1 \upto F_k]$, so
  \[
  \overline{g}^{s q} \in \phi^*\left(K[F_1 \upto F_k]\right) = A.
  \]
  This shows that indeed $\widehat{A} = K[X]^G$. Since $K[X]^G$ is
  finitely generated as a $K$-algebra (see \mycite{nag:d}) and $K[X]^G
  = \widehat{A}$ by the above, $K[X]^G$ is finitely generated as an
  $A$-module. This proves that \aref{2aInvarAff} terminates after
  finitely many steps.
\end{proof}

\begin{problem}
  We are still left with the problem of finding an algorithm that
  computes $A^G$, where $A$ is a finitely generated $K$-algebra which
  need not be reduced and $G$ is a reductive group acting on $A$ such
  that $A$ is locally finite. By \mycite{nag:d}, $A^G$ is finitely
  generated in this case.
\end{problem}

\subsection{Connected groups acting on normal varieties} \label{3sNormal}

In this section we consider the case of a connected reductive group
$G$ acting on a normal, irreducible affine variety $X$. This case is
more special than the one dealt with in \aref{2aInvarAff}. But we will
present a simpler and probably faster algorithm for computing
$K[X]^G$. The idea for this algorithm was stimulated by the
paper~[\citenumber{Hashimoto}] of \citename{Hashimoto}, which gives an
algorithm for computing generating invariants of a simply connected
simple linear algebraic group with a linear action.

Recall that for a reductive group $G$ and a $G$-module $V$ we can
always compute a subalgebra $A \subseteq K[V]^G$ such that $K[V]^G$ is
integral over $A$. Indeed, the possibly simplest way of doing this is
by computing what \mycite{kem.separating} calls the ``Derksen ideal''
by performing the first two steps of Algorithm~2.9
in~[\citenumber{kem.separating}] (same as the first three steps in
Algorithm~4.1.9 from~[\citenumber{Derksen:Kemper}]), and then setting
one set of variables equal to zero in the generators of the Derksen
ideal (Step~4 in [\citenumber{Derksen:Kemper}, Algorithm~4.1.9]). This
will yield a set of polynomials $\{g_1 \upto g_s\} \subset K[V]$ which
define Hilbert's nullcone (see [\citenumber{Derksen:Kemper},
Section~2.4.1 and Remark~4.1.4]). Now use Algorithm~2.7
from~[\citenumber{kem.separating}] to compute homogeneous invariants
$f_1 \upto f_k \in K[V]^G$ degree by degree until every $g_i$ lies in
the radical of the ideal in $K[V]$ generated by the $f_j$. Then
$K[V]^G$ will be integral over $K[f_1 \upto f_k]$. An alternative
method would be to use Algorithm~2.9
from~[\citenumber{kem.separating}] to compute a graded separating
subalgebra of $K[V]^G$. Then $K[V]^G$ will be integral over this
subalgebra (see Lemma~1.3 in [\citenumber{kem.separating}]). Compared
with the first method outlined above, computing separating invariants
involves one additional major Gr\"obner basis computation, which is
not really necessary for our purposes.

We can now present an algorithm for computing $K[X]^G$ for $X$ normal
and $G$ connected and reductive. The algorithm involves the
computation of the integral closure of one ring in another, which will
be discussed shortly.

\begin{alg}[Invariants for $G$ connected and reductive, $X$ normal]
  \label{3aNormal} \mbox{}
  \begin{description}
  \item[Input:] A connected, reductive group $G$ over an algebraically
    closed field $K$, and a normal, irreducible $G$-variety $X$, given
    according to \conref{1cInput}.
  \item[Output:] Generators of $K[X]^G$ as a $K$-algebra.
  \end{description}
  \begin{enumerate}
    \renewcommand{\theenumi}{\arabic{enumi}}
  \item \label{s51} Use \aref{1aEmbed} to calculate an equivariant
    embedding $\map{\phi}{X}{V}$ into a $G$-module $V$.
  \item \label{s52} Construct invariants $f_1 \upto f_k \in K[V]^G$
    such that $K[V]^G$ is integral over $K[f_1 \upto f_k]$ (see the
    above discussion).
  \item \label{s53} Form the subalgebra $A \subseteq K[X]^G$ generated
    by all $f_i \circ \phi$ (see Step~\ref{s43} of \aref{2aInvarAff}).
  \item \label{s54} Use \aref{3aIC} to compute the integral closure
    $B$ of $A$ in $K[X]$. Then
    \[
    K[X]^G = B.
    \]
  \end{enumerate}
\end{alg}

The following lemma will be used in the proof of correctness of
\linebreak \aref{3aNormal}. We write $G^0$ for the connected component
of an algebraic group $G$.

\begin{lemma} \label{3lConnected}
  Let $G$ be an affine algebraic group over an algebraically closed
  field $K$, and let $X$ be a $G$-variety. Let $A \subseteq K[X]^G$ be
  a subalgebra such that $K[X]^G$ is integral over $A$. Then
  $K[X]^{G^0}$ is the integral closure of $A$ in $K[X]$.
\end{lemma}

\begin{proof}
  We write $B$ for the integral closure of $A$ in $K[X]$. First take
  $b \in B$ arbitrary. There exists a monic polynomial $F \in A[T]$
  with $F(b) = 0$. Thus for every $\sigma \in G$ we also have
  $F\left(\sigma(b)\right) = 0$. On the other hand, $F$ has at most
  finitely many zeros in $K[X]$. Indeed, this follows from the fact
  that for each irreducible component $X_i$ of $X$, restricting the
  coefficients of $F$ yields a non-zero polynomial with only finitely
  many zeros in $K[X_i]$. It follows that the $G$-orbit of~$b$ is
  finite.  Therefore the stabilizer $G_b \subseteq G$ of~$b$ has
  finite index in $G$, which implies $G^0 \subseteq G_b$. Hence $b \in
  K[X]^{G^0}$.

  Conversely, take $f \in K[X]^{G^0}$. Then
  \[
  F(T) := \prod_{\sigma \in G/G^0} \left(T - \sigma(f)\right) \in
  K[X]^G[T],
  \]
  and $F(f) = 0$. So $f$ is integral over $K[X]^G$ and hence also over
  $A$. It follows that $f \in B$.
\end{proof}

\begin{proof}[Proof of correctness of \aref{3aNormal}]
  It follows from the reductivity of $G$ that $K[X]^G$ is integral
  over $A$. From this, $K[X]^G = B$ follows by \lref{3lConnected}.
\end{proof}

The following algorithm for computing the integral closure of one ring
in another is mostly drawn from \mycite[Chapter~6]{Vasconcelos}.

\begin{alg}[Integral closure] \label{3aIC} \mbox{}
  \begin{description}
  \item[Input:] A prime ideal $I \subseteq K[x_1 \upto x_n]$ defining
    a normal domain $B := \linebreak K[x_1 \upto x_n]/I$, and
    polynomials $f_1 \upto f_k \in K[x_1 \upto x_n]$ defining a
    subalgebra $A = K[\overline{f_1} \upto \overline{f_k}] \subseteq
    B$, where we write $\overline{f_i} := f_i + I$.
  \item[Output:] Polynomials $g_1 \upto g_r \in K[x_1 \upto x_n]$ such
    that $K[\overline{g_1} \upto \overline{g_r}]$ is the integral
    closure of $A$ in $B$.
  \end{description}
  \begin{enumerate}
    \renewcommand{\theenumi}{\arabic{enumi}}
  \item \label{s61} With an additional indeterminate~$t$, form the
    algebra
    \[
    D := K[\overline{f_1} \upto \overline{f_k},t,t \overline{x_1}
    \upto t \overline{x_n}] \subseteq B[t].
    \]
  \item \label{s62} Compute $h_1 \upto h_r \in K[x_1 \upto x_n,t]$
    such that the $\overline{h_i} \in B[t]$ generate the normalization
    $\tilde{D}$ of $D$. This can be done by using de Jong's algorithm
    (see \mycite{deJong} or \mycite[Section~1.6]{Derksen:Kemper}).
  \item \label{s63} For $i = 1 \upto r$, obtain $g_i$ from by setting
    $t = 0$ in $h_i$.
  \end{enumerate}  
\end{alg}

\begin{proof}[Proof of correctness of \aref{3aIC}]
  Since $B$ is a normal domain, the same is true for $B[t]$ (see, for
  example, \mycite[Exercise~4.18]{eis}). Therefore $\tilde{D}$ is
  contained in $B[t]$, which shows that its generators
  $\overline{h_i}$ do lie in $B[t]$ rather than just in $Q(B)[t]$,
  where $Q(B)$ denotes the field of fractions of $B$. Consider the map
  $\map{\phi}{B[t]}{B}$ of $B$-algebras given by $t \mapsto 0$. The
  definition of $D$ implies $\phi(D) = A$. For each $h \in \tilde{D}$
  we have an equation
  \[
  h^s + d_1 h^{s-1} + \cdots + d_{r-1} h + d_s = 0
  \]
  with $d_i \in D$. Applying $\phi$ to this yields an integral
  equation for $\phi(h)$ over $A$. If follows that the $\overline{g_i}
  = \phi(\overline{h_i})$ from Step~\ref{s63} are integral over $A$.
  
  Conversely, take $g \in B$ arbitrary such that $b$ is integral over
  $A$. Then $g$, seen as an element of $B[t]$, is integral over $D$.
  Moreover, $Q(D) = Q(B[t])$ by the definition of $D$, so $g \in
  Q(D)$. It follows that $g \in \tilde{D}$, so there exists a
  polynomial $F$ such that $g = F\left(\overline{h_1} \upto
    \overline{h_r}\right)$. Applying $\phi$ yields
  \[
  g = \phi(g) = F\left(\overline{g_1} \upto \overline{g_r}\right).
  \]
  This completes the proof.
\end{proof}
\begin{remark} \label{1rNonNormal}
In \aref{3aIC} we have assumed that $B$ is normal. 
We will sketch how to deal with the more general case
where $B$ is a domain which need not be normal.
 Compute
the normalization $\widetilde{B}$ of $B$ using De Jong's algorithm
(see \mycite{deJong} or \mycite[Section~1.6]{Derksen:Kemper}).
Let $\widetilde{A}$ be the integral closure of $A$ in $\widetilde{B}$.
Generators of $\widetilde{A}$ can be computed using \aref{3aIC}.
Find $A$-module generators $h_1,\dots,h_s$ of $\widetilde{A}$. 
Define
$$
M=\{(a_1,\dots,a_s)\in B^{s}\mid \sum_{i=1}^s a_i h_i\in B\}.
$$
Find $g\in B \setminus \{0\}$ such that $gh_i\in B$ for all $i$.
We may identify $M$ with
\begin{multline*}
  \{(a_1,\dots,a_s,b)\in B^{s+1}\mid \sum_{i=1}^s a_i h_i+b=0 \}= \\
  \{(a_1,\dots,a_s,b)\in B^{s+1}\mid \sum_{i=1}^s a_i g h_i+bg=0 \}.
\end{multline*}
So $M$ can be viewed as a syzygy module, and generators of $M$ can be computed
using  \mycite[\S1.3]{Vasconcelos} or \mycite[\S1.3]{Derksen:Kemper}
(computing syzygies between elements $u_1+I,\dots,u_t+I$
in $B=K[x_1,\dots,x_n]/I$ can easily be reduced
to computing syzygies between $u_1,\dots,u_t$ and generators of $I$ in
the polynomial ring $K[x_1,\dots,x_n]$).
We have
$$
M\cap A^s=\{(a_1,\dots,a_s)\in A^s\mid \sum_{i=1}^s a_ih_i\in \widetilde{A}\cap B\}.
$$
Define $\phi:M\to B$ by $\phi(a_1,\dots,a_r)=\sum_{i=1}^s a_ih_i$. Then $\phi(M\cap A^s)=\widetilde{A}\cap B$ is the integral closure of $A$ in $B$.
Generators of $M\cap A^s$ can be computed, using Remark~\ref{2remmodintersect}.
\end{remark}

\section{Quasi-affine varieties and Hilbert's fourteenth problem}
\label{2sQuasiAffine}

This section provides some methods for dealing with non-finitely
generated algebras. 
\subsection{The colon operation} \label{2sColon}
For a subset $B$ of a ring, $B^r$ will denote the set of all products
of $r$ elements from $B$.
We generalize the notion of a colon ideal as follows.
\begin{defi}
For a commutative ring $S$ and subsets $A,B\subseteq S$ we define
$$
(A:B)_S=\{f\in S\mid fB\subseteq A\}
$$
and
$$
(A:B^{\infty})_S=\bigcup_{r=1}^\infty (A:B^r)_{S}=\{f\in S\mid \exists r\ 
fB^r\subseteq A\}.
$$
\end{defi}
\begin{example}
If ${\mathfrak a}$ and ${\mathfrak b}$ are ideals of $S$, then
$({\mathfrak a}:{\mathfrak b})_S$ and $({\mathfrak a}:{\mathfrak b}^\infty)_S$
are the usual colon ideals (see for example \mycite[Chapter~2]{Vasconcelos}).
\end{example}
If $R$ is a domain with quotient field $Q(R)$, and $f\in R\setminus\{0\}$ then
$$
(R:\{f\}^{\infty})_{Q(R)}=R_f,
$$
the localization of $R$ with respect to the element $f$.
This generalizes as follows. Suppose that $R=K[X]$ is the coordinate ring of an irreducible affine variety
$X$.
Let $Y\subseteq X$ be a the zero set of an ideal ${\mathfrak a}\subseteq R$.
The ring of regular functions on the
quasi-affine variety $U:=X\setminus Y$ is denoted by 
$K[U]$.
\begin{lemma} \label{2lQuasiAffine}
We have
$$
K[U]=(R:{\mathfrak a}^{\infty})_{Q(R)}.
$$
\end{lemma}
\begin{proof}
Suppose that $f\in (R:{\mathfrak a}^\infty)_{Q(R)}$ and $p\in U$.
There exists $h\in {\mathfrak a}$ with $h(p)\neq 0$. We have
$g=h^sf\in R$ for some nonnegative integer $s$. So $f=h^{-s}g$
is a regular function on an open neighborhood of $p\in U$. Since $p\in U$
was chosen arbitrarily, we conclude that $f\in K[U]$.

Conversely, suppose that $f\in K[U]$. We may write ${\mathfrak
a}=(a_1,\dots,a_r)$. Because $K[U]\subseteq R_{a_i}$ there exists
a nonnegative integer $l_i$ such that $a_i^{l_i}f\in R$ for all $i$.
Set $N=l_1+l_2+\cdots+l_r-r+1$. Then
$$
{\mathfrak a}^Nf\subseteq R,
$$
because ${\mathfrak a}^N$ is spanned by monomials
$a_1^{k_1}a_2^{k_2}\cdots a_r^{k_r}$ with $k_1+\cdots+k_r=N$
and the definition of $N$ implies that $k_i\geq l_i$ for some $i$.
\end{proof}
If $f\in {\mathfrak a}$ is nonzero, then we have $K[U]\subseteq R_f$ and
\begin{equation} \label{2eqRf}
  K[U]=(R:{\mathfrak a}^\infty)_{Q(R)}=(R:{\mathfrak a}^{\infty})_{R_f}.
\end{equation}
Note that such a ring of regular functions on a quasi-affine variety
is not always finitely generated over $K$ (see
\mycite[Chapter~V.5]{nag:b} or \mycite{winkelmann}). Rings of the form
$(R:{\mathfrak a}^\infty)_{Q(R)}$ are {\it ideal transforms\/} in the
sense of \mycite{nag:b}.  Suppose that $G$ is an algebraic group and
$X$ is an affine $G$-variety.  Nagata showed that the invariant ring
$K[X]^G$ may not be finitely generated~[\citenumber{nag:a}]. However,
he also showed that if $X$ is normal, then the invariant ring $K[X]^G$
is isomorphic to some {\it ideal transform\/} of a finitely generated
domain over $K$~[\citenumber{nag:b}, Chapter~V, Proposition~4]. In
other words, $K[X]^G$ can be viewed as $K[U]$ for some quasi-affine
variety $U$.  Later, we will study this in more detail.

The following lemma is easy to prove:
\begin{lemma}\ 
\begin{enumerate}
\item If ${\mathfrak a}$ is an ideal of the ring $S$,
and $B\subseteq S$ then
$({\mathfrak a}:B)_S$ and $({\mathfrak a}:B^\infty)_S$ are ideals
of $S$.
\item If $S$ is an algebra over some field,
$A\subseteq S$ is a subalgebra and $B\subseteq A$, then
$(A:B^\infty)_S$
is a subalgebra of $S$.
\end{enumerate}
\end{lemma}
Suppose that the additive group $\Ga$ acts regularly
on an irreducible affine variety $X$. Then $\Ga$ also acts on the coordinate ring $S:=K[X]$. An algorithm for computing
the generators of the invariant ring $S^{\Ga}$ was given
by \mycite{essen}. Van den Essen first constructs a subalgebra $R$ of
the invariant ring, and an element $f\in R$ such that $S^{\Ga}=R_f\cap
S=(R:f^{\infty})_S$ (for details, see \sref{3lMu}). He then gives an algorithm for computing a set of
generators of the ring $S^{\Ga}=(R:f^{\infty})_S$ over $K$. The
algorithm terminates if this ring is finitely generated.

In this section we will give a generalization of Van den Essen's
algorithm for computing generators of $(R:f^{\infty})_S$. 
We will give an algorithm for computing generators of
the ring $(R:{\mathfrak a}^{\infty})_S$ for a finitely generated domain
$S$ over $K$, a finitely generated subalgebra $R$ and any ideal ${\mathfrak a}$ of $R$.
Our algorithm will terminate if and only if $(R:{\mathfrak a}^\infty)_S$
is finitely generated. This extension is quite useful, as it 
allows us to compute rings of regular
functions on  irreducible quasi-affine varieties by using~\eqref{2eqRf}.

Suppose that $S$ is a domain over a field $K$, 
$R$ is a finitely generated
subalgebra and ${\mathfrak a}\subseteq R$ is an ideal. Then
$(R:{\mathfrak a})_S$ is an $R$-module. Suppose that ${\mathfrak a}$ is nonzero.
Then we can choose a nonzero element $f\in {\mathfrak a}$.
From the definition it follows that
$f(R:{\mathfrak a})_S\subseteq R$. This way, we may identify
$(R:{\mathfrak a})_S$ as a submodule of $R$. In particular,
$(R:{\mathfrak a})_S$ is finitely generated as an $R$-module.
We will first give an algorithm for finding $R$-module generators
of $(R:{\mathfrak a})_S$.

\begin{con}
We assume that $S=K[x_1,\dots,x_n]/I$ where $I$ is a prime ideal
generated by a finite set ${\mathcal G}_I$.
\end{con}
\begin{alg}[Computation of $(R:{\mathfrak a})_S$]\mbox{}\label{alg4.7}
\begin{description}
\item[Input:] Polynomials $f_1,\dots,f_r\in K[x_1,\dots,x_n]$
such that $R$ is generated by $f_1+I,\dots,f_r+I$, and a finite set 
${\mathcal A}\subset K[y_1,y_2,\dots,y_r]$ such
that the (nonzero) ideal ${\mathfrak a}\subseteq R$ is generated by 
$g(f_1,\dots,f_r)+I$, 
$g\in {\mathcal A}$.
\item[Output:] A finite set ${\mathcal H}\subseteq K[x_1,\dots,x_n]$
such that $(R:{\mathfrak a})_S$ is generated by $1+I$
and all $h+I$, $h\in {\mathcal H}$
as an $R$-module. Moreover, if $(R:{\mathfrak a})_S=R$ then ${\mathcal H}=\emptyset$.

\end{description}
\begin{enumerate}
 \renewcommand{\theenumi}{\arabic{enumi}}
\item Let ${\mathfrak b}$ be the ideal in
  $K[x_1,\dots,x_n,y_1,\dots,y_r]$ generated by $I$ and all $y_i-f_i$,
  $i=1,\dots,r$. Compute a Gr\"obner basis ${\mathcal G}_J$ of
  $J:={\mathfrak b}\cap K[y_1,\dots,y_r]$.  (Choose an elimination
  ordering on the monomials of $K[x_1,\dots,x_n,\linebreak
  y_1,\dots,y_r]$ and compute a Gr\"obner basis ${\mathcal
    G}_{\mathfrak b}$ of ${\mathfrak b}$.  Then we have ${\mathcal
    G}_{J}={\mathcal G}_{\mathfrak b}\cap K[y_1,\dots,y_r]$.)
\item Choose $u\in {\mathcal A}$ such that 
$u\not\in J$. (Reduce all elements $u\in
{\mathcal A}$ with respect to the Gr\"obner basis ${\mathcal G}_{J}$
until we have found an element $u$ that
does not reduce to $0$.)
\item Let ${\mathfrak d}\subseteq K[y_1,\dots,y_r]$ be the ideal 
generated by $J$ and ${\mathcal A}$.
Compute a Gr\"obner basis ${\mathcal G}_{\mathfrak c}$
of the colon ideal ${\mathfrak c}:=(J+(u)):{\mathfrak d}$.
\item Let ${\mathfrak v}\subseteq K[x_1,\dots,x_n,y_1,\dots,y_r]$ be the 
ideal generated by $I$, $u$ and all $y_i-f_i$, $i=1,2,\dots,r$.
Compute a Gr\"obner basis ${\mathcal G}_{\mathfrak u}$
of ${\mathfrak u}:={\mathfrak v}\cap K[y_1,\dots,y_r]$.
\item Compute  a 
Gr\"obner basis ${\mathcal G}_{\mathfrak q}$ of the
intersection ${\mathfrak q}:={\mathfrak u}\cap {\mathfrak c}$.

\item Compute a Gr\"obner basis 
${\mathcal G}_{\mathfrak p}$ of the ideal ${\mathfrak p}:=J+(u)$
in $K[y_1,\dots,y_r]$.

\item Replace ${\mathcal G}_{\mathfrak q}$ by the subset of all elements
that do not reduce to $0$ with respect to the Gr\"obner basis ${\mathcal
G}_{\mathfrak p}$.

\item If ${\mathcal G}_{\mathfrak q}=\{v_1,\dots,v_s\}$, compute
$h_1,\dots,h_s\in K[x_1,\dots,x_n]$ such that
$$v_i(f_1,\dots,f_r)+I=u(f_1,\dots,f_r)h_i+I$$ 
for all $i$. 
To find $h_1,\dots,h_s$, proceed as follows.
Each $v_i$ can be expressed in the form
$$
v_i=\sum_{g\in {\mathcal G}_I}a_{i,g} g+b_iu+\sum_{j}c_{i,j}(y_j-f_j).
$$
with $a_{i,g},b_i,c_{i,j}\in K[x_1,\dots,x_n,y_1,\dots,y_r]$
for all $g,i,j$.
(this can be done using the extended Gr\"obner basis algorithm in step~(4)).
Then plug in $y_i=f_i$ for all $i$. We take 
$$h_i=b_i(x_1,\dots,x_n,f_1,\dots,f_r)
$$ 
for all
$i$. Set ${\mathcal H}=\{h_1,\dots,h_s\}$.
\end{enumerate}
\end{alg}
\begin{proof}[Proof of correctness of Algorithm~\ref{alg4.7}.]
Consider the ring homomorphism
$$
\phi:K[y_1,\dots,y_r]\to K[x_1,\dots,x_n]/I\cong S
$$
defined by $y_i\mapsto f_i+I$. The image of $\phi$ is isomorphic to $R$,
and the kernel of $\phi$ is $J$. So we have
$$
K[y_1,\dots,y_r]/J\cong R.
$$
The ideal ${\mathfrak a}\subseteq R$ is generated by all 
$\phi(g)$, $g\in {\mathcal A}$.
Since ${\mathfrak a}\subseteq R$ is a nonzero ideal, there must exist
a $u\in {\mathcal A}$ such that $\phi(u)\neq 0$. Hence
there exists a $u\in {\mathcal A}$ that does not reduce
to $0$ modulo ${\mathcal G}_J$.
The colon ideal $(\phi(u)R:{\mathfrak a})_R\subseteq R$
is equal to $\phi({\mathfrak c})$, and
$\phi^{-1}((\phi(u)R:{\mathfrak a})_R)={\mathfrak c}$.
The ideal ${\mathfrak u}$ is equal to $\phi^{-1}(\phi(u)S)$.
We have
$${\mathfrak q}=\phi^{-1}((\phi(u)R:{\mathfrak a})_R\cap \phi^{-1}(\phi(u)S)
=\phi^{-1}((\phi(u)R:{\mathfrak a})_R\cap \phi(u)S).
$$
Also, we get
$$
{\mathfrak p}=\phi^{-1}(\phi(u)R)=(u)+J.
$$
After step (7), ${\mathfrak q}$
is generated as an ideal in $R$ by ${\mathcal G}_{\mathfrak q}$, $u$ and $J$.
It follows that $(\phi(u) R:{\mathfrak a})_R\cap \phi(u)S$ is generated by $\phi(h),h\in {\mathcal G}_q$ and $\phi(u)$.

Since 
$$
\varphi(u)(R:{\mathfrak a})_S=(\varphi(u)R:{\mathfrak a})_R\cap \varphi(u)S,
$$
we have that $(R:{\mathfrak a})_S$ is generated as an $R$-module by
$1=\phi(u)/\phi(u)$ and all $\phi(v)/\phi(u)$, $v\in {\mathcal G}_q$.
If ${\mathcal G}_q=\{v_1,\dots,v_s\}$ then
$$
\phi(v_i)=\phi(u)(h_i+I)
$$
for all $i$.
Since ${\mathcal H}=\{h_1,\dots,h_s\}$ we have
that $(R:{\mathfrak a})_S$ is generated by all $1+I$ and all $h+I$, $h\in
{\mathcal H}$.

By step (6) and (7) we have that $\phi(v_i)\not\in \phi(u)R$,
and $h_i+I\not\in R$.
Hence, if $(R:{\mathfrak a})_S=R$ then ${\mathcal H}=\emptyset$.
\end{proof}
\begin{alg}[Computation of $(R:{\mathfrak a}^\infty)_S$]\label{alg4.8}\mbox{}
\begin{description}
\item[Input:] Polynomials $f_1,\dots,f_r\in K[x_1,\dots,x_n]$
such that $R$ is generated by $f_1+I,\dots,f_r+I$, and a finite set 
${\mathcal A}\subset K[y_1,y_2,\dots,y_r]$ such
that the (nonzero) ideal ${\mathfrak a}\subseteq R$ is generated by 
$g(f_1,\dots,f_r)+I$, 
$g\in {\mathcal A}$.
\item[Output:] A (possibly infinite) sequence
$h_1,h_2,h_3,\dots$ of elements
in \linebreak $K[x_1,\dots,x_n]$ such that $h_1+I,h_2+I,\dots$
generate $(R:{\mathfrak a}^\infty)_S$ as a $K$-algebra. If
$(R:{\mathfrak a}^\infty)$ is finitely generated, then
the algorithm will terminate after finite time and 
the output will be a finite sequence.
\end{description}
\begin{enumerate}
 \renewcommand{\theenumi}{\arabic{enumi}}
 \item $F=\emptyset$
 \item ${\mathcal H}=\{f_1,\dots,f_r\}$
 \item {\tt while} ${\mathcal H}\neq \emptyset$ {\tt do}
 \item ${\tt output}(\mathcal H)$
 \item $F:=F\cup {\mathcal H}$.
 \item Let ${\mathcal H}$ be the output of Algorithm~\ref{alg4.7}
 for the computation of $(\widetilde{R}:{\mathfrak a})_S$,
 where $\widetilde{R}$ is the algebra generated by all
 $f+I$, $f\in F$ and  
 ${\mathfrak a}$ is the ideal in $\widetilde{R}$
 generated by all $g(f_1,\dots,f_r)+I$, $g\in {\mathcal A}$.
 \item {\tt enddo}
 \end{enumerate}
 \end{alg}

\begin{proof}[Proof of correctness of Algorithm~\ref{alg4.8}.]
Let $\widetilde{R}_i$ be the algebra $\widetilde{R}$ in step~(6)
in the $i$-th iteration of the {\tt while} loop 
in lines (3)--(7). We have $\widetilde{R}_1=R$ and 
$$
\widetilde{R}_{i+1}\supseteq(\widetilde{R}_{i}:{\mathfrak a}\widetilde{R}_i)_S=
(\widetilde{R}_i:{\mathfrak a})_S.
$$
where ${\mathfrak a}$ is the ideal in $R$ generated by all $g(f_1,\dots,f_r)+I$,
$g\in {\mathcal A}$. It easily follows by induction that
$\widetilde{R}_{i+1}\supseteq(R:{\mathfrak a}^i)_S$ for all $i$.
 Note that in step~(6), the algebra $\widetilde{R}_i$
is generated by all $h+I$ with $h\in F$. Moreover, $F$ is exactly
the set of all polynomials that have been sent to the output.

If the algorithm does not terminate, then we have
$$
\widetilde{R}_1\subseteq \widetilde{R}_2\subseteq \cdots
$$
and
$$
(R:{\mathfrak a}^\infty)_S=
\bigcup_{i=1}^\infty (R:{\mathfrak a}^i)_S\subseteq\bigcup_{i=1}^\infty 
\widetilde{R}_i.
$$
On the other hand it is easy to see (by induction) 
that $\widetilde{R}_i\subseteq (R:{\mathfrak a}^\infty)_S$
for all $i$. It follows that
\begin{equation}\label{unionofrings}
(R:{\mathfrak a}^\infty)_S=\bigcup_i \widetilde{R}_i.
\end{equation}
If the output is $h_1,h_2,\dots$ then the
algebra generated by $h_1+I,h_2+I,\dots$ contains $\widetilde{R}_i$
for all $i$. Therefore, the algebra generated by $h_1+I,h_2+I,\dots$
is $(R:{\mathfrak a}^\infty)_S$.

Suppose that $(R:{\mathfrak a}^\infty)_S$ is finitely generated.
By (\ref{unionofrings}),  $\widetilde{R}_i$ contains all generators of $(R:{\mathfrak a}^\infty)_S$ for some $i$,
and $\widetilde{R}_i=(R:{\mathfrak a}^\infty)_S$.
But then ${\mathcal H}=\emptyset$ after the $i$-th iteration
of the while loop and the algorithm terminates.
The output is exactly $F$ and $\widetilde{R}_i=(R:{\mathfrak a}^\infty)_S$
is generated by all $h+I$, $h\in F$.
\end{proof}
%
%
  \subsection{Finite generation}
  In this section we study domains which are not finitely generated over $K$. We introduce
  the {\it finite generation locus ideal\/} of such an algebra.
\begin{prop}\label{6ploc}
Suppose that $S$ is a domain which is finitely generated over a field $K$
and that $R$ is a subalgebra of $S$. Then there exists an nonzero element $f\in R$
such that $R_f$ is finitely generated as a $K$-algebra.
\end{prop}
\begin{proof}
Choose a finitely generated subalgebra $T\subseteq R$ such that $T$ and $R$
have the same quotient field. By the theorem of generic freeness  (see \mycite[Theorem~14.4]{eis}
or \rref{2rGenFree} below), there exists
a nonzero element $f\in T$ such that $S_f$ is a free $T_f$-module. Let
$B$ be a basis of $S_f$ over $T_f$. We can write
$$
1=\sum_{i=1}^r u_ie_i
$$
with $e_1,e_2 \upto e_r \in B$ and $u_1,u_2,\dots,u_r\in
T_f$.
Since $R_f$ and $T_f$ have the same quotient field, it follows that the submodule $R_f\subseteq S_f$
is contained in
$$
T_fe_1\oplus T_fe_2\oplus \cdots \oplus T_fe_r\cong T_f^r.
$$
This shows that $R_f$ is contained in a finitely generated $T_f$-module. Since $T_f$ is a finitely generated algebra, $R_f$ is finitely generated as a $T_f$-module. It follows
that $R_f$ is a finitely generated algebra.
\end{proof}
The following result is well-known. We give a proof for the reader's
convenience.

\begin{prop} \label{propfg}
  Suppose that $R$ is a domain over $K$ and $f,g \in R \setminus
  \{0\}$ such that $(f,g) = R$. If $R_f$ and $R_g$ are finitely
  generated, then so is $R$, and $R = R_f \cap R_g$.
\end{prop}

\begin{proof}
  We may write $R_f = K[a_1 \upto a_r,f^{-1}]$ and $R_g = K[b_1 \upto
  b_l,g^{-1}]$ with $a_i,b_j \in R$. We have $1 = x f + y g$ with $x,y
  \in R$. Take $z \in R_f \cap R_g$. Then
  \[
  z = \frac{a}{f^m} = \frac{b}{g^n} \quad \text{with} \quad n,m \in
  \NN,\ a \in K[a_1 \upto a_r,f],\ \text{and} \ b \in K[b_1 \upto
  b_l,g],
  \]
  so
  \begin{multline*}
    z = z (x f + y g)^{m+n} = \\
    \sum_{i=1}^m \binom{m+n}{i} (x f)^i y^{m+n-i} g^{m-i} b +
    \sum_{i=m+1}^{m+n} \binom{m+n}{i} x^i f^{i-m} (y g)^{m+n-i} a
  \end{multline*}
  Thus
  \[
  R \subseteq R_f \cap R_g \subseteq K[a_1 \upto a_r,b_1 \upto
  b_l,f,g,x,y] \subseteq R,
  \]
  and the result follows.
\end{proof}

\begin{prop}
For a domain $R$ defined over a field $K$, define 
$$
{\mathfrak g}=\{0\}\cup \{f\in R\setminus\{0\}\mid R_f\mbox{ is a finitely generated $K$-algebra}\}.
$$
Then ${\mathfrak g}$ is a radical ideal of $R$.
\end{prop}
\begin{proof}
If $f\in {\mathfrak g}$ and $g\in R$ are both nonzero, then
$$
R_{fg}=(R_f)_g
$$
is finitely generated, because $R_f$ is finitely generated. This implies $fg\in {\mathfrak g}$.

Suppose $f,g\in {\mathfrak g}$ such that $f$, $g$, and $f+g$ are all
non-zero.
We have $(f,g)R_{f+g} = R_{f+g}$, and the algebras
$(R_{f+g})_f=(R_f)_{f+g}$ and $(R_{f+g})_g=(R_g)_{f+g}$ are finitely generated.
By Proposition~\ref{propfg}, $R_{f+g}$ is finitely generated, so
$f+g\in {\mathfrak g}$. It follows that $\mathfrak g$ is an ideal.

The ideal ${\mathfrak g}$  is clearly a radical ideal 
since $R_{f^r}=R_f$ for every $f\in R$ and any positive integer $r$.
\end{proof}
We will call ${\mathfrak g}$ the {\it finite generation locus ideal\/} of $R$.
Note that ${\mathfrak g}=R$ if and only if $R$ is finitely generated.
If $R$ is a subalgebra of a finitely generated algebra, then the
finite generation locus ideal is nonzero by Proposition~\ref{6ploc}.

\begin{lemma}\label{6lcolon}
Suppose that $S$ is a domain over $K$, $R$ is a subalgebra,
and ${\mathfrak a}\subseteq R$ is an ideal.
Set ${\mathfrak b}=(R:(R:{\mathfrak a})_S)_S$. Then ${\mathfrak b}$ is an ideal of $R$, and
${\mathfrak a}\subseteq {\mathfrak b}$. Moreover,
$$
(R:{\mathfrak a}^i)_S=(R:{\mathfrak b}^i)_S
$$
for $i \in \NN \cup \{\infty\}$.
\end{lemma}
\begin{proof}
Since ${\mathfrak a}(R:{\mathfrak a})_S\subseteq R$ by definition
of $(R:{\mathfrak a})_S$ we get ${\mathfrak a}\subseteq {\mathfrak b}:=(R:(R:{\mathfrak a})_S)_S$.
Since $1\in (R: {\mathfrak a})_S$ we get ${\mathfrak b}=(R:(R:{\mathfrak a})_S)_S\subseteq
(R:\{1\})_S=R$. Also, ${\mathfrak b}$ is clearly an $R$-module, so it is an ideal of
$R$.
%
Since ${\mathfrak a}\subseteq {\mathfrak b}$ we  have 
$$
(R:{\mathfrak a})_S\supseteq (R:{\mathfrak b})_S.
$$
Because ${\mathfrak b}=(R:(R:{\mathfrak a})_S)_S$,
we get ${\mathfrak b}(R:{\mathfrak a})_S\subseteq R$.
From this it follows that
$$
(R:{\mathfrak a})_S\subseteq (R:{\mathfrak b})_S.
$$
We conclude that
$$
(R:{\mathfrak a})_S=(R:{\mathfrak b})_S.
$$
By induction on $i$ we prove that
$$
(R:{\mathfrak a}^i)_S=(R:{\mathfrak b}^i)_S.
$$
The case $i=1$ has already been done. Suppose that $i>1$.
Then we have
\begin{multline*}
  (R:{\mathfrak a}^i)_S=((R:{\mathfrak a})_S:{\mathfrak a}^{i-1})_S=
  ((R:{\mathfrak b})_S:{\mathfrak a}^{i-1})_S= \\
  (R:{\mathfrak b}{\mathfrak a}^{i-1})_S= ((R:{\mathfrak
    a}^{i-1})_S:{\mathfrak b})_S.
\end{multline*}
By induction we may assume that $(R:{\mathfrak a}^{i-1})_S=(R:{\mathfrak b}^{i-1})_S$.
So we get
$$
(R:{\mathfrak a}^i)_S=((R:{\mathfrak a}^{i-1})_S:{\mathfrak b})_S=
((R:{\mathfrak b}^{i-1})_S:{\mathfrak b})_S=
(R:{\mathfrak b}^i)_S.
$$
We also have
$$
(R:{\mathfrak a}^\infty)_S=\bigcup_i (R:{\mathfrak a}^i)_S=\bigcup_i(R:{\mathfrak b}^i)_S=
(R:{\mathfrak b}^\infty)_S
$$
\end{proof}
 \begin{lemma}\label{6lfgli}
 Suppose that $R$ is a finitely generated subalgebra of a domain $S$ over a field $K$, ${\mathfrak a}$
 is an ideal of $R$ and
 suppose that $\widetilde{R}=(R:{\mathfrak
   a}^\infty)_S=\bigcup_i\widetilde{R}_i$, where
 $$
 R \subseteq \widetilde{R}_1\subseteq \widetilde{R}_2\subseteq \cdots
 $$
 is a sequence of finitely generated $K$-algebras.
 Define the ideal ${\mathfrak g}_i$ of $\widetilde{R}_i$ by
 $$
 {\mathfrak g}_i=\sqrt{(\widetilde{R}_i:(\widetilde{R}_i:{\mathfrak a})_S)_S},
 $$
 where the radical ideal is taken in $\widetilde{R}_i$.
 Then we have
 $$
 {\mathfrak g}_1\subseteq {\mathfrak g}_2\subseteq \cdots
 $$
 and
  $$
{\mathfrak g}:= \bigcup_i {\mathfrak g}_i
 $$
is the finite generation locus ideal of $\widetilde{R}$.
 \end{lemma}
 \begin{proof}
 Let us define ${\mathfrak h}_i=(\widetilde{R}_i:(\widetilde{R}_i:{\mathfrak a})_S)_S$ so
 that ${\mathfrak g}_i=\sqrt{{\mathfrak h}_i} $.
Note that 
\begin{equation} \label{2eqRi}
  \widetilde{R}=(R:{\mathfrak a}^\infty)_S=(\widetilde{R}_i:{\mathfrak
    a}^\infty)_S= (\widetilde{R}_i:{\mathfrak
    h}_i^\infty)_S=(\widetilde{R}_i:{\mathfrak g}_i^\infty)_S
\end{equation}
 by Lemma~\ref{6lcolon}. Let $u_1,u_2,\dots,u_t$ be
 generators of the $\widetilde{R}_{i+1}$-module 
 $(\widetilde{R}_{i+1}:{\mathfrak a})_S$.
 This module is contained in
 $\widetilde{R}=(\widetilde{R}_i:{\mathfrak g}_i^\infty)_S$.
 Therefore, there  exists a positive integer $l$ such that
 $$
 {\mathfrak g}_i^lu_j\subseteq \widetilde{R}_i
 $$
 for all $j$. It follows that
  $$
 {\mathfrak g}_i^l(\widetilde{R}_{i+1}:{\mathfrak a})_S\subseteq \widetilde{R}_{i+1}
 $$
 and
 $$
 {\mathfrak g}_i^l\subseteq (\widetilde{R}_{i+1}:(\widetilde{R}_{i+1}:{\mathfrak a})_S)_S=
 {\mathfrak h}_{i+1}.
 $$
 Taking radicals  on both sides gives us
 $$
 {\mathfrak g}_i\subseteq \sqrt{{\mathfrak h}_{i+1}}={\mathfrak g}_{i+1}.
 $$

We now show that ${\mathfrak g}=\bigcup_{i} {\mathfrak g}_i$ is the
finite generation locus ideal of $\widetilde{R}$.
 If $f\in {\mathfrak g} \setminus
\{0\}$, then $f\in {\mathfrak g}_i$ for some $i$.
We have
  $$\widetilde{R}= (\widetilde{R}_i:{\mathfrak g}_i^\infty)_S\subseteq (\widetilde{R}_i)_f,
 $$
 because $f\in {\mathfrak g}_i$.
 It follows that
 $$
 \widetilde{R}_f=(\widetilde{R}_i)_f
 $$
 is finitely generated.
 
 Conversely, suppose that $\widetilde{R}_f$ is finitely generated for
 some $f\in \widetilde{R} \setminus \{0\}$.
 Say, $\widetilde{R}_f$ is generated over $K$ by $h_1,h_2,\dots,h_r\in \widetilde{R}$ and $1/f$.
 For some $i$, we have $f,h_1,h_2,\dots,h_r\in \widetilde{R}_i$.
 Therefore, we get
 $$
 \widetilde{R}\subseteq (K[f,h_1,\dots,h_r]:f^\infty)_S\subseteq (\widetilde{R}_i)_f
$$
Since $(\widetilde{R}_i:{\mathfrak a})_S$ is a finitely generated
$\widetilde{R}_i$-module, there exists a positive integer $l$ such that
$$
f^l(\widetilde{R}_i:{\mathfrak a})_S\subseteq \widetilde{R}_i.
$$
We see that
$$
f^l\in (\widetilde{R}_i:(\widetilde{R}_i:{\mathfrak a})_S)_S={\mathfrak h}_i.
$$
and $f\in {\mathfrak g}_i$.
\end{proof}

Using Lemma~\ref{6lfgli}, it is now possible to find generators of the
finite generation locus ideal of the ring $(R:{\mathfrak
  a}^\infty)_S$.  To do this, we modify Algorithm~\ref{alg4.8} as
follows.

\begin{alg}\label{6afgli}
An algorithm for finding generators of the
finite generation locus ideal of an algebra of the form $(R:{\mathfrak a}^\infty)_S$
where $S = K[x_1 \upto x_n]/I$ is a finitely generated domain over a field $K$, $R$ is a finitely generated subalgebra
of $S$ and ${\mathfrak a}$ is an ideal of $R$.

\begin{description}
\item[Input:] Polynomials $f_1,\dots,f_r\in K[x_1,\dots,x_n]$
such that $R$ is generated by $f_1+I,\dots,f_r+I$, and a finite set 
${\mathcal A}\subset K[y_1,y_2,\dots,y_r]$ such
that the (nonzero) ideal ${\mathfrak a}\subseteq R$ is generated by 
$g(f_1,\dots,f_r)+I$, 
$g\in {\mathcal A}$.
\item[Output:] A (possible infinite) sequence
$h_1,h_2,h_3,\dots$ of elements
in \linebreak $K[x_1,\dots,x_n]$ such that $h_1+I,h_2+I,\dots$
generate the finite generation locus ideal ${\mathfrak g}$
of $(R:{\mathfrak a}^\infty)_S$.
\end{description}
\begin{enumerate}
 \renewcommand{\theenumi}{\arabic{enumi}}
 \item $F:=\emptyset$
 \item ${\mathcal H}:=\{f_1,\dots,f_r\}$
 \item {\tt while} ${\mathcal H}\neq \emptyset$ {\tt do}
  \item $F:=F\cup {\mathcal H}$.
   \item {\tt output} generators of
 $\widetilde{\mathfrak g}:=\sqrt{(\widetilde{R}:(\widetilde{R}:{\mathfrak a})_S)_S}$
 where $\widetilde{R}$ is the $K$-algebra generated by all $f+I$, $f\in F$,
 and ${\mathfrak a}$ is the ideal in $\widetilde{R}$
 generated by all $g(f_1,\dots,f_r)+I$, $g\in {\mathcal A}$.
  \item Let ${\mathcal H}$ be the output of Algorithm~\ref{alg4.7}
 for the computation of $(\widetilde{R}:{\mathfrak a})_S$.
  \item {\tt enddo}
 \end{enumerate}
 The algorithm terminates if and only if $(R:{\mathfrak a}^\infty)_S$
 is finitely generated. In that case ${\mathfrak g}$ is the
 whole ring $(R:{\mathfrak a}^\infty)_S$. So the interesting
 case is when the algorithm does not terminate. One should
 add a termination criterion in step~(3), i.e., replace step (3) by
 
 \noindent{\tt while ${\mathcal H}\neq \emptyset$ and not
 {\rm [\sc termination criterion]} do},
 
 \noindent where [{\sc termination criterion}] is some criterion. 
 For example, one could allow at most $k$ iterations of the
 loop (3)--(7) where $k$ is a parameter given in the input.
 Another example of a possible termination criterion will be given
 in Algorithm~\ref{6afgli2}. 
  
To compute generators of $\widetilde{\mathfrak g}$ in step (5), one proceeds as follows. 
We compute generators of $(\widetilde{R}:{\mathfrak a})_S$ using \aref{alg4.7}.
Let
$$
{\mathfrak h}:=(\widetilde{R}:(\widetilde{R}:{\mathfrak a})_S)_S
$$
Choose a nonzero element $f\in {\mathfrak a}$.
Since $1\in (\widetilde{R}:{\mathfrak a})_S$ we have 
$$
{\mathfrak h}=(\widetilde{R}:(\widetilde{R}:{\mathfrak a})_S)_{\widetilde{R}}=
(f\widetilde{R}:f(\widetilde{R}:{\mathfrak a})_S)_{\widetilde{R}},
$$
so generators of ${\mathfrak h}$ can be computed because
it is again a colon ideal. Finally, generators of $\widetilde{\mathfrak g}$ can be computed
using an algorithm to compute the radical ideal of ${\mathfrak h}$ (see for example 
\mycite[Section~1.5]{Derksen:Kemper}, \mycite{Matsumoto:01}, or \mycite{kem:radical}).
The correctness of the algorithm follows from Lemma~\ref{6lfgli}.
\end{alg}
\subsection{Hilbert's fourteenth problem}
Suppose that $K$ is a field, $L$ is a subfield of the rational
function field \linebreak $K(x_1,x_2,\dots,x_n)$ containing $K$.
Hilbert's $14^{\rm th}$ problem asks whether $L\cap \linebreak
K[x_1,\dots,x_n]$ is finitely generated. Nagata gave a counterexample
to this conjecture~[\citenumber{nag:a}]. In fact, Nagata constructed
an algebraic (non-reductive) group $G$ and a linear action of $G$ on
the polynomial ring such that $K[x_1,\dots,x_n]^G$ is not finitely
generated. If we take $L=K(x_1,\dots,x_n)^G$ as the invariant field,
then $L\cap K[x_1,\dots,x_n]=K[x_1,\dots,x_n]^G$ is not finitely
generated, so this gives indeed a counterexample to Hilbert's
fourteenth problem.  It is not clear whether it is decidable whether
$L\cap K[x_1,\dots,x_n]$ is finitely generated, or even whether $L\cap
K[x_1,\dots,x_n]=K$.

We will replace $K[x_1,\dots,x_n]$ by an arbitrary finitely generated
domain $S$ over $K$. Let $L$ be a subfield of the quotient
field $Q(S)$ of $S$. We assume that $L$ is generated as a field
 by elements of the ring $S$. In other words, $L$ is the quotient
 field of some subalgebra $R\subseteq S$. We will present an
 algorithm to compute generators of  the algebra $L\cap S=Q(R)\cap S$.
 This algorithm will terminate if  this algebra is finitely generated. 
 First we need the following constructive version
 of ``generic freeness'':
\begin{theorem}\label{propgenfree}
 Suppose that $S$ is a finitely generated domain over $K$, and
 $R$ is a finitely generated subalgebra, then there exists
 an algorithm that finds a nonzero element $f\in R$ such that
 $S_f$ is a free $R_f$-module, and $R_f$ is a direct summand of $S_f$.
 \end{theorem}
 See \mycite[Theorem~14.4]{eis} for a proof of a more general version of 
 Grothendieck's generic freeness lemma. Note that this lemma is often
 called ``generic flatness'', but that almost all proofs found in the literature prove
 the stronger ``generic freeness'' property.
 We will give here an algorithm to find the $f$ in question. 
 For a slightly different algorithm, see \mycite[Theorem~2.6.1]{Vasconcelos}.
 We assume
 that $K$ is a field for which we have algorithms
 for a zero test and all arithmetic operations.
  Assume that $S=R[x_1,\dots,x_r]/I$ where $x_1,\dots,x_r$ are indeterminates.

\begin{alg}[Generic Freeness]\ \label{5agenfree}

\begin{description}
\item{\bf Input:} $R$, $S$, generators of $I$.
\item{\bf Output:} An element $f\in R\setminus\{0\}$ such that $S_f$ is a free $R_f$-module,
and $R_f$ is a direct summand in $S_f$.
\end{description}
 \begin{enumerate}  
 \renewcommand{\theenumi}{\arabic{enumi}} 
 \item Let $J$ be the ideal in $Q(R)[x_1,\dots,x_r]$ generated by $I$  (so it has the
 same set of generators as $I$).
 \item Compute a Gr\"obner basis ${\mathcal G}$ of $J$ with respect to some monomial
 ordering. If necessary, multiply the polynomials from $\mathcal G$ by
 constants from $Q(R)$ to make their leading coefficients equal to~1.
 \item Compute $f\in R \setminus \{0\}$ such that $fh(x_1,\dots,x_r)\in R[x_1,\dots,x_r]$ for every $h(x_1,\dots,x_r)\in {\mathcal G}$.
 \end{enumerate}
 \end{alg}

 \begin{proof}[Proof of correctness of Algorithm~\ref{5agenfree}]
 Let
 $$
 \varphi:R[x_1,\dots,x_r]\to S
 $$
 be the homomorphism with kernel $I$ that induces an isomorphism
 \linebreak $R[x_1,\dots,x_r]/I \cong S$.  Let $M$ be the set of all
 monomials $m$ such that $m$ is not divisible by any leading monomial
 $\lm(h)$ with $h\in {\mathcal G}$.  We claim that $S_f$ is a free
 $R_f$-module with basis $\varphi(M)$.
  
  Suppose that $h\in S_f$. There exists a positive integer $l$ such that $f^lh\in S$. 
  We can write $f^lh=u(x_1,\dots,x_r)+I$ where $u(x_1,\dots,x_r)\in
  R[x_1,\dots,x_r]\subseteq Q(R)[x_1,\dots,x_r]$.
  Let $v(x_1,\dots,x_r)$ be the normal form of $u(x_1,\dots,x_r)$ with
  respect to the Gr\"obner basis ${\mathcal G}$.
  Thus if 
  $${\mathcal G}=\{h_1(x_1,\dots,x_r),\dots,h_s(x_1,\dots,x_r)\},$$
  then there exist $a_1(x_1,\dots,x_r),\dots,a_s(x_1,\dots,x_r)\in
  Q(R)[x_1,\dots,x_r]$ such that
  $$
  u(x_1,\dots,x_r)-v(x_1,\dots,x_r)=\sum_{i=1}^s a_i(x_1,\dots,x_r)h_i(x_1,\dots,x_r).
  $$
  Note that $h_1(x_1,\dots,x_r),\dots,h_s(x_1,\dots,x_s)\in R_f[x_1,\dots,x_r]$. If $p(x_1,\dots,x_r)\in R_f[x_1,\dots,x_r]$
  and $q(x_1,\dots,x_r)$ is obtained from $p(x_1,\dots,x_r)$ by a single reduction step modulo the
  Gr\"obner basis ${\mathcal G}$, then $q(x_1,\dots,x_r)\in R_f[x_1,\dots,x_r]$ as well.
  From this observation one can show using induction
  that
  \[
  a_1(x_1,\dots,x_r),\dots,a_s(x_1,\dots,x_r),v(x_1,\dots,x_r)\in
  R_f[x_1,\dots,x_r].
  \]
 Now we get $v(x_1,\dots,x_r)\in R_fM$, $\varphi(v(x_1,\dots,x_r))=f^lh\in R_f\varphi(M)$ and
 $h\in R_f\varphi(M)$.
 This shows that $S_f=R_f\varphi(M)$, i.e., $\varphi(M)$ generates $S_f$ as an $R_f$ module.
 It is clear from Gr\"obner basis theory that $\varphi(M)$ is
 a linearly independent set over  $Q(R)$.
 We conclude that $S_f$ is a free $R_f$ module with basis $\varphi(M)$.
 We can identify $R_f$ with $R_f\varphi(1)=R_f\cdot 1\subseteq S_f$,
 which is a direct summand because
 $$
 S_f=R_f\cdot 1\oplus R_f\cdot \varphi(M\setminus\{1\}).
 $$
    \end{proof}

\begin{rem} \label{2rGenFree}
  \aref{5agenfree} is also correct in the case where $R$ is not 
  finitely generated. The only problem is that we cannot provide a way
  of computing the ideals $I$ and $J$ in this case. In fact, it in not
  even clear how to compute with elements from $Q(R)$ if $R$ is not
  finitely generated. Nevertheless, the above proof of correctness of
  the algorithm does provide a proof of the generic freeness theorem
  even for $R$ not finitely generated.
\end{rem}

 \begin{alg}[Intersection of a field and a finitely generated domain]\ \label{5aHilb14}
 
 \begin{description}
  \item[\bf Input:] Generators and relations for a finitely generated domain  $S$ over $K$
 and generators of a finitely generated subalgebra $R$.
 \item[\bf Output:] Generators of the algebra $Q(R)\cap S$. The algorithm will
 terminate if $Q(R)\cap S$ is finitely generated. If $Q(R)\cap S$ is not
 finitely generated, then the algorithm will not terminate but the (infinite) output
 will still generate the algebra $Q(R)\cap S$.
 \end{description}
 \begin{enumerate}
 \renewcommand{\theenumi}{\arabic{enumi}} 
 \item Use \aref{5agenfree} to compute $f\in R \setminus \{0\}$ such that $R_f$
 is a  summand in the $R_f$-module $S_f$.
 \item Compute generators of $(R:f^\infty)_S$ using
   Algorithm~\ref{alg4.8}.
    \end{enumerate}
 \end{alg}
 \begin{proof}[Proof of correctness of Algorithm~\ref{5aHilb14}]
 We can write 
 $$
 S_f=R_f\oplus C
 $$
 where $C$ is an $R_f$-module. Let $\pi:S_f\to R_f$ be the projection onto $R_f$.
 So $\pi$ is an $R_f$-module homomorphism such that $\pi(a)=a$ if and only if $a\in R_f$.
 Suppose that $s=a/b\in S_f$ with $a,b\in R_f$. Then we have $bs=a$ and
 $b\pi(s)=\pi(bs)=\pi(a)=a$. So we obtain $s=a/b=\pi(s)\in R_f$.
 This shows that $Q(R)\cap S_f=R_f$.
  It follows that 
 $$Q(R)\cap S\subseteq R_f\cap S=(R:f^\infty)_S,$$
 so $Q(R)\cap S=(R:f^\infty)_S$ because the other inclusion is trivial.
 \end{proof}
 The following theorem is Proposition 4 in Chapter V of
 \mycite{nag:b}.
\begin{theorem}\label{6tquasiaffine}
Suppose that $R$ is a finitely generated normal domain over a field $K$,
and $L$ is a subfield of $Q(R)$ containing $K$. Then $R\cap L$ is isomorphic
to the ring of regular functions on some quasi-affine variety $U$ defined over $K$.
In other words, there exists a finitely generated domain $T$ over $K$ and
an ideal ${\mathfrak a}$ of $T$ such that
$$R\cap L=(T:{\mathfrak a}^\infty)_{Q(T)}.$$
\end{theorem}
Some extensions of this result can be found in \mycite{winkelmann}.
Theorem~\ref{6tquasiaffine} inspires us to ask the following questions.
\begin{problem}
Let $R$ and $L$ be as in Theorem~\ref{6tquasiaffine}. Find an algorithm
to construct generators of $T$ and ${\mathfrak a}$ where $T$ and ${\mathfrak a}$
are as in Theorem~\ref{6tquasiaffine}. 
\end{problem}
\begin{problem}\label{6pcolonquasiaffine}
Suppose that $S$ is a finitely generated normal domain over $K$, $R$
is a finitely generated 
normal subalgebra 
and ${\mathfrak a}$ is an ideal of $R$. Is the ring $(R:{\mathfrak a}^\infty)_S$
isomorphic to the ring of regular functions on some quasi-affine variety over $K$?
\end{problem}

The following proposition gives a positive answer to
Problem~\ref{6pcolonquasiaffine} under an additional hypothesis. We
will later see that this hypothesis is satisfied in a situation which
is of interest in invariant theory (see \aref{7afactorial}).

\begin{prop} \label{6lCodim}
Suppose that $S,R,{\mathfrak a}$ are as in Problem~\ref{6pcolonquasiaffine}.
Let ${\mathfrak g}$ be the finite generation locus ideal of
$(R:{\mathfrak a}^\infty)_S$. 
Suppose that the affine variety corresponding to the ideal ${\mathfrak g}S$ has
codimension $\geq 2$,  in other words, all prime ideals containing
  ${\mathfrak g} S$ have height $\ge 2$. 
Then $(R:{\mathfrak a}^\infty)_S$ is isomorphic
to the coordinate ring of an quasi-affine variety. 
\end{prop}
\begin{proof}
The proposition follows from the correctness of the algorithm below.
\end{proof}
The following algorithm is a modification of \aref{6afgli}.
\begin{alg}\label{6afgli2}
An algorithm for finding 
a subalgebra $\widetilde{R}\subseteq S$ and an ideal $\widetilde{\mathfrak g}$
of $\widetilde{R}$ such that
$$
(R:{\mathfrak a}^\infty)_S=(\widetilde{R}:\widetilde{\mathfrak g}^\infty)_{Q(\widetilde{R})},
$$
where $S$ is a finitely generated normal domain over $K$, 
$R$ is a finitely generated subalgebra, and
${\mathfrak a}$ is an ideal of $R$, such that
the affine variety corresponding to ${\mathfrak g}S$ has codimension at least $2$,
where ${\mathfrak g}$ is the finite generation locus ideal of
$(R:{\mathfrak a}^\infty)_S$.
\begin{description}
\item[Input:] Polynomials $f_1,\dots,f_r\in K[x_1,\dots,x_n]$
such that $R$ is generated by $f_1+I,\dots,f_r+I \in
K[x_1,\dots,x_n]/I =: S$, and a finite set 
${\mathcal A}\subset K[y_1,y_2,\dots,y_r]$ such
that the (nonzero) ideal ${\mathfrak a}\subseteq R$ is generated by 
$g(f_1,\dots,f_r)+I$, 
$g\in {\mathcal A}$.
\item[Output:] Generators of a subalgebra $\widetilde{R}$ of $R$
and generators of an ideal $\widetilde{\mathfrak g}$ of $\widetilde{R}$
such that
$$
(R:{\mathfrak a}^\infty)_S=(\widetilde{R}:\widetilde{\mathfrak g}^\infty)_{Q(\widetilde{R})}
$$
\end{description}
\begin{enumerate}
 \renewcommand{\theenumi}{\arabic{enumi}}
 \item Set $F:=\emptyset$ and $\widetilde{\mathfrak g} := \{0\}$.
 \item ${\mathcal H}:=\{f_1,\dots,f_r\}$.
 \item \label{2sCrit} {\tt while ${\mathcal H}\neq \emptyset$ 
 and [{\rm 
 $\widetilde{\mathfrak g}S$ has codimension
$< 2$}] and $\widetilde{\mathfrak g} S \ne S$ do}
  \item $F:=F\cup {\mathcal H}$.
     \item compute generators of
 $\widetilde{\mathfrak g}:=\sqrt{(\widetilde{R}:(\widetilde{R}:{\mathfrak a})_S)_S}$
 where $\widetilde{R}$ is the $K$-algebra generated by all $f+I$, $f\in F$,
 and  ${\mathfrak a}$ is the ideal in $\widetilde{R}$
 generated by all $g(f_1,\dots,f_r)+I$, $g\in {\mathcal A}$. The
 radical ideal is meant to be formed in $\widetilde{R}$.
  \item Let ${\mathcal H}$ be the output of Algorithm~\ref{alg4.7}
 for the computation of $(\widetilde{R}:{\mathfrak a})_S$.

 \item {\tt enddo}
 \item output generators of $\widetilde{R}$ and $\widetilde{\mathfrak g}$
 \end{enumerate}
 \end{alg}

\begin{rem*}
  In step~\eqref{2sCrit} of the algorithm, it is easy to determine the
  codimension of $\widetilde{\mathfrak g} S$, since by
  \mycite[Corollary~13.4]{eis}, the codimension equals $\dim(S) -
  \dim\left(S/\widetilde{\mathfrak g} S\right)$. The dimension can be
  read off a Gr\"obner basis, see
  \mycite[Corollary~7.5.5]{Greuel.Pfister} or
  \mycite[Section~1.2.5]{Derksen:Kemper}.
  
  We also remark that the ideal $\widetilde{\mathfrak g}$ found by the
  algorithm is not necessarily the finite generation locus ideal.
\end{rem*}

\begin{proof}[Proof of correctness proof of \aref{6afgli2}]

Let $\widetilde{R}_i$ and ${\mathfrak g}_i$
 be the algebra $\widetilde{R}$ and the ideal $\widetilde{\mathfrak g}$
in the $i$-th iteration of loop (3)--(7).
We have
$$\widetilde{R}_1\subseteq \widetilde{R}_2\subseteq\cdots$$
and
$$
{\mathfrak g}_1\subseteq {\mathfrak g}_2\subseteq\cdots
$$
such that ${\mathfrak g}_i$ is an ideal of $\widetilde{R}_i$ for all $i$.

Assume that the algorithm does not terminate and the loop (3)-(7) is repeated infinitely
many times.
Then
$\bigcup_{i}\widetilde{R}_i=(R:{\mathfrak a}^\infty)_S$ 
and 
${\mathfrak g}=\bigcup_i{\mathfrak g}_i$
is the finite generation locus ideal of $(R:{\mathfrak a}^\infty)_S$,
because of the correctness of Algorithm~\ref{6afgli}.
So we have
$$
{\mathfrak g}_1S\subseteq {\mathfrak g}_2S\subseteq {\mathfrak g}_3S\subseteq\cdots.
$$
Since $S$ is finitely generated over $K$, it is Noetherian.
There exists an index $k$ such that
$$
{\mathfrak g}_kS={\mathfrak g}_{k+1}S=\cdots=\bigcup_i{\mathfrak g}_iS={\mathfrak g}S.
$$
In particular, there exists an index $k$ such that the affine variety
corresponding to the ideal ${\mathfrak g}_kS$ has codimension $\geq 2$.
Let $k$ be minimal with this property. This implies that the algorithm
terminates after the $k$-th iteration of the loop (3)-(7), and the output is $\widetilde{R}_k$
and ${\mathfrak g}_k$.

Let $X$ be the affine variety such that $S=K[X]$. 
If $f\in (S:{\mathfrak g}_k^{\infty})_{Q(S)}$, then $f$ is a rational function on $X$
which is regular on all of $X$ except for a closed subset of codimension $\geq 2$.
Since $X$ is normal, $f$ is regular on $X$ (see \mycite[below Corollary 11.4]{eis}), i.e., $f\in S$.
This shows that
 $$
(S:{\mathfrak g}_k^\infty)_{Q(S)}=S.
$$
So we have
$$
(\widetilde{R}_k:{\mathfrak g}_k^\infty)_{Q(\widetilde{R}_k)}\subseteq
(S:{\mathfrak g}_k^\infty)_{Q(S)}=S.
$$
It follows that
$$
(\widetilde{R}_k:{\mathfrak g}_k^\infty)_{Q(\widetilde{R}_k)}=
(\widetilde{R}_k:{\mathfrak g}_k^\infty)_S=(R:{\mathfrak a}^\infty)_S,
$$
where the last equality follows from~\eqref{2eqRi}.
\end{proof}

\section{Invariant rings of algebraic groups} \label{3sAlgebraic}
Suppose that $K$ is an algebraically closed field (of arbitrary characteristic)
and $G$ is an algebraic group over $K$ which acts regularly on an
affine variety $X$. If $G$ is not reductive, then $K[X]^G$ may
not be finitely generated.
\begin{problem}
Find an algorithm which determines whether $K[X]^G$ is finitely generated.
\end{problem}
\begin{problem}
Given that $K[X]^G$ is finitely generated, find an algorithm that
computes a set of generators for $K[X]^G$.
\end{problem}
If $G$ is reductive, then $K[X]^G$ is known to be finitely generated and an
algorithm was given in \sref{1sReductive}. If $G$ is the additive group and
the characteristic of the ground field is $0$, then an algorithm
was given by \mycite{essen}.
Here we will give such an algorithm in arbitrary characteristic and where
$G$ can be any connected unipotent group.

Even if $K[X]^G$ is not finitely generated, there are still interesting questions
to ask. Let $K(X)^G$ be the field of invariant rational functions on $X$. Then we have
$$
K[X]^G=K[X]\cap K(X)^G,
$$
If $X$ is normal, then there exists a {\it quasi-affine\/} variety $U$ over $K$
such that
$$
K[X]^G=K[U]
$$
by Theorem~\ref{6tquasiaffine}.
\begin{problem}
Find an algorithm which constructs a quasi-affine variety $U$ such that
$K[X]^G=K[U]$.
\end{problem}
We will give such an algorithm where $G$ is a connected unipotent group and
$K[X]$ is a unique factorization domain.
\subsection{Invariants of the additive group}\label{7sadditive}
Suppose that $G=\Ga$ is the additive group
acting regularly on an irreducible affine variety $X$ over an algebraically
closed field $K$. The coordinate ring $K[\Ga]$ can be identified
with the polynomial ring $K[t]$. The group addition $\Ga\times \Ga\to \Ga$
corresponds to a ring homomorphism $K[t]\to K[t]\otimes K[t]$ defined
 by $t\mapsto t\otimes 1+1\otimes t$.
 The action $\Ga\times X\to X$ corresponds to a ring homomorphism
$$
\mu:K[X]\to K[\Ga\times X]\cong K[\Ga]\otimes K[X]\cong K[X][t].
$$
Suppose that $f\in K[X]$. We can write
$$
\mu(f)=f_0+f_1t+f_2t^2+\cdots+f_rt^r
$$
with $f_0,\dots,f_r\in K[X]$.
If $\sigma\in \Ga$, then we have
$$
((-\sigma)\cdot f)(x)=f(\sigma \cdot x)=\mu(f)(\sigma,x)=f_0(x)+f_1(x)\sigma+\cdots+f_r(x)\sigma^r,
$$
so
$$
((-\sigma)\cdot f)=f_0+f_1\sigma+\cdots+f_r\sigma^r.
$$
In particular we have
\begin{equation} \label{3eqF0}
  f=0\cdot f=f_0.
\end{equation}
We have
\begin{multline}\label{eq31}
  (\tau-\sigma)\cdot
  f=f_0+f_1(\sigma-\tau)+\cdots+f_r(\sigma-\tau)^r= \\
  (\tau)\cdot ((-\sigma)\cdot f)= (\tau\cdot f_0)+(\tau\cdot
  f_1)\sigma+\cdots +(\tau\cdot f_r)\sigma^r
\end{multline}
for all $\sigma,\tau\in \Ga\cong K$. Comparing
the coefficients of $\sigma^r$ shows that  $\tau\cdot f_r=f_r$ for all $\tau\in \Ga$.
This implies that $f_r\in K[X]^{\Ga}$.
We may extend $\mu$ to be defined for all $f=g/h$ with $g \in K[X]$
and
$h \in K[X]^{\Ga}$ by setting $\mu(f) = \mu(g)/h$. Then~\eqref{eq31}
still holds.

If the action of $\Ga$ is trivial, then
of course $K[X]^{\Ga}=K[X]$. So let us assume that $\Ga$ acts non-trivially.
Then there exists an $f\in K[X]$ such that $\mu(f)\neq f$. This
element~$f$ will be chosen once and fixed for the rest of \sref{7sadditive}.
We can write
$$
F(t):=\mu(f)=f_0+f_1t+\cdots+f_{r-1}t^{r-1}+f_rt^r, \quad t\in \Ga,
$$
with $r>0$ and $f_r\neq 0$.

\subsubsection{Characteristic 0 case.} \label{3lMu}
If $K$ has characteristic $0$, then an algorithm
was given by Van den Essen for computing generators of $K[X]^{\Ga}$. 
This algorithm terminates if $K[X]^{\Ga}$ is finitely generated.
We will sketch the idea behind this algorithm. 
We set $s=f_{r-1}/(rf_r)$.
From the coefficient of $\sigma^{r-1}$ in (\ref{eq31}), it follows 
that $\tau\cdot f_{r-1}=f_{r-1}-rf_{r}\tau$
and $\tau\cdot s=s-\tau$ for all $\tau\in \Ga$.
\begin{lemma}
If $h\in K[X]_{f_r}$, then $\mu(h)\mid_{t=-s}\in
K[X]_{f_r}^{\Ga}$.
\end{lemma}
\begin{proof}
  Set
  \[
  H(t) := \mu(h) = h_0 + h_1 t + \cdots + h_l t^l
  \]
  with $h_i \in K[X]_{f_r}$.
  From (\ref{eq31}) it follows that
$$
H(t-\tau)=h_0+h_1(t-\tau)+\cdots+h_l(t-\tau)^l=(\tau\cdot h_0)+(\tau\cdot h_1)t+\cdots (\tau\cdot h_l)t^l
$$
Using this for $t=-s+\tau$ gives us
$$
\tau\cdot H(-s)=(\tau\cdot h_0)+(\tau\cdot h_1)(-\tau\cdot s)+\cdots +(\tau\cdot h_l)(-\tau\cdot s)^l=$$
$$=
(\tau\cdot h_0)+(\tau\cdot h_1)(-s+\tau)+\cdots +(\tau\cdot h_l)(-s+\tau)^l=H((-s+\tau)-\tau)=H(-s).
$$
\end{proof}

Suppose that $K[X]=K[h_1,\dots,h_m]$. Define 
$$
g_i=\mu(h_i)\mid_{t=-s}\in K[X]_{f_r}^{\Ga}
$$
for $i=1,2,\dots,m$.
For every $i$, choose a natural number $k_i$ such that $u_i:=f_r^{k_i}g_i\in
K[X]$.
\begin{lemma}\label{7lGa}
We have
$$
K[X]^{\Ga}_{f_r}=K[g_1,\dots,g_m,1/f_r]=K[u_1,\dots,u_m,1/f_r].
$$
\end{lemma}
\begin{proof}
  Define the ring homomorphism $\gamma:K[X]_{f_r}\to K[X]_{f_r}^{\Ga}$
  by $\gamma(g)= \linebreak \mu(g)\mid_{t=-s}$.  The homomorphism
  $\gamma$ is surjective, because $\gamma(g)=\mu(g)\mid_{t=-s}=g$ for
  all $g\in K[X]_{f_r}^{\Ga}$.  Since $K[X]_{f_r}$ is generated by
  $h_1,\dots,h_m,1/f_r$, $K[X]^{\Ga}_{f_r}$ is generated by
  $\gamma(h_1)=g_1,\dots,\gamma(h_m)=g_m,\gamma(1/f_r)=1/f_r$.
\end{proof}
From Lemma~\ref{7lGa} it follows that
$$
K[X]^{\Ga}=(K[u_1,\dots,u_m,f_r]:(f_r)^\infty)_{K[X]}
$$
Now generators of $K[X]^{\Ga}$ can be computed using Algorithm~\ref{alg4.8}.

\subsubsection{Arbitrary characteristic.} \label{3sGap}
Let us now no longer assume that $K$ has characteristic 0. Van den Essen's
algorithm may not work because $r$ may be divisible by the characteristic
of $K$ for every possible choice of $f$, as the following example shows.
\begin{example}
Suppose that $K$ is a field of characteristic 2 and define an
action of the additive group on $K[x,y]$ by
$$
\mu(x)=x+ty+t^2, \mu(y)=y
$$
For every element $f\in K[x,y]$, $\mu(f)$ is a polynomial
of even degree in $t$.
\end{example}

Let $X$ be an irreducible affine variety on which $\Ga$ acts regularly and non-trivially.
Choose again $f\in K[X]$ such that $\mu(f)\neq f$. Again we can write
$$
F(t):=\mu(f)=f_0+f_1t+\cdots+f_rt^r
$$
with $r>0$ and $f_r\neq 0$.
\begin{lemma}\label{7lfg}
If $f_r=1$, then $K[X]^{\Ga}$ is finitely generated.
\end{lemma}
\begin{proof}
  Suppose an invariant $g \in K[X]^{\Ga}$ is mapped to zero by the
  canonical map $\map{\pi_f}{K[X]^{\Ga}}{K[X]/(f)}$. Then $g = h f$
  with $h \in K[X]$, so
  \[
  g = \mu(g) = \mu(h) F(t)
  \]
  This implies $g = 0$, since otherwise the degrees of both sides of
  the above equation would differ. It follows that $\pi_f$
induces an inclusion $K[X]^{\Ga}\to K[X]/(f)$.

We claim that $K[X]/(f)$ is integral over $K[X]^{\Ga}$.
Suppose that $u\in K[X]$ and let $U(t)=\mu(u)\in K[X][t]$. Define $P(s)\in K[X][s]$
as the resultant
$$
P(s)=\Res_t(U(t)-s,F(t)).
$$
Since $F(t)$ is monic, it is clear from the definition of
the resultant as the determinant of the Sylvester matrix (see \mycite[IV, \S8]{Lang}) that
either $P(s)$ or $-P(s)$ is monic as well. 

Consider the action of $\Ga$ on $K[X][t,s]$, where $\Ga$ acts trivially on the variables $t,s$.
If $\sigma\in \Ga$, then 
$\sigma\cdot U(t)=U(t-\sigma)$ by \eqref{eq31}, and similarly
$\sigma\cdot F(t)=F(t-\sigma)$. Therefore
$$
\sigma\cdot P(s)=\Res_{t}(U(t-\sigma)-s,F(t-\sigma))=\Res_t(U(t)-s,F(t))=P(s)
$$
using \mycite[Proposition 8.3]{Lang}. It follows that
 all coefficients of $P(s)$ lie in $K[X]^{\Ga}$.

There exist polynomials
$A(t,s),B(t,s)\in K[X][t,s]$ such that
$$
P(s)=A(t,s)(U(t)-s)+B(t,s)F(t)
$$
(see \mycite[discussion before IV, Proposition 8.1]{Lang}). If we
substitute $t = 0$ and $s = u$, we get
\[
P(u) = A(0,u) (U(0) - u) + B(0,u) F(0) = B(0,u) f,
\]
where the last equality follows from~\eqref{3eqF0}.
Therefore $P(u+(f))=0$ in $K[X]/(f)$, so $u+(f)$ is integral over $K[X]^{\Ga}$. 
The monic polynomial among $P(x),-P(x)$ is the {\it characteristic polynomial}
of $u+(f)$ over $K[X]^{\Ga}$.
Since $u$ was arbitrary,
$K[X]/(f)$ is integral over $K[X]^{\Ga}$.

Suppose that $h_1,\dots,h_m$ are generators of $K[X]$.
Let $R\subseteq K[X]^{\Ga}$ be the subalgebra generated
by the coefficients of the characteristic polynomials
of $h_i+(f)\in K[X]/(f)$
over $K[X]^{\Ga}$ for $i=1,2,\dots,m$. 
We have $R\subseteq K[X]^{\Ga}\subseteq K[X]/(f)$
and $R$ is clearly finitely generated.
Since $K[X]/(f)$ is finitely generated and integral over $R$,
we have that $K[X]/(f)$ is a finite $R$ module.
Since $K[X]^{\Ga}$ is a sub-$R$-module of $K[X]/(f)$, it
is finitely generated as an $R$-module as well.
But then $K[X]^{\Ga}$ is also finitely generated as an algebra.
\end{proof}

If $f_r=1$ and $X$ is normal, then generators of $K[X]^{\Ga}$ can be computed
as follows. By \lref{3lConnected}, $K[X]^{\Ga}$ is the integral closure of $R$ in $K[X]$, where
$R$ is as in the proof of Lemma~\ref{7lfg}.
This integral closure can be computed as described in
Algorithm~\ref{3aIC}. If $f_r = 1$ but $X$ is not normal,
\rref{1rNonNormal} may be applied to compute the integral closure.

 Let us now consider the general case where $f_r$ need not be~1 and $X$
 need not be normal (but is still assumed to be irreducible). Let ${\mathfrak s}\subseteq K[X]$
 be the vanishing ideal of the singular locus. This ideal is non-zero
 and stable under 
 the action of $\Ga$. Without loss of generality, we could have
 chosen $f\in {\mathfrak s}$ such that $\mu(f)\neq f$.
 We write 
 $$
 F(t)=\mu(f)=f_0+f_1t+\cdots+f_rt^r
 $$
 with $f_r\neq 0$.
 Choose distinct $\lambda_0,\lambda_1,\dots,\lambda_r\in K$.
 Using that the Vandermonde matrix is invertible, we see that $f_0,f_1,\dots,f_r$ lie in the $K$-linear span of
 $F(\lambda_0),F(\lambda_1),\dots,F(\lambda_r)$.
 We have $F(\lambda_0),\dots,F(\lambda_r)\in {\mathfrak s}$
 because ${\mathfrak s}$ is $\Ga$-stable. This implies that $f_r\in {\mathfrak s}$. So $f_r$ vanishes on the set of singularities,
 and $K[X]_{f_r}$ is smooth.
   We have
$$
\mu(f/f_r)=(f_0/f_r)+(f_1/f_r)t+\cdots+(f_{r-1}/f_r)t^{r-1}+t^{r}.
$$
Using the previous discussion we can compute generators 
of $K[X]_{f_r}^{\Ga}$. Of course there is no need to choose~$f$ to lie
in $\mathfrak s$ if we apply \rref{1rNonNormal} to compute the
integral closure. Suppose that
$$
K[X]_{f_r}^{\Ga}=K[g_1,\dots,g_l]
$$
For every $i$ we can compute a nonnegative integer $k_i$ such
that $u_i:=f_r^{k_i}g_i\in K[X]$.
We then have
$$
K[X]^{\Ga}=(K[u_1,\dots,u_l,f_r]:(f_r)^\infty)_{K[X]}.
$$
Now generators of $K[X]^{\Ga}$ can be computed using
Algorithm~\ref{alg4.8}.

\subsection{Invariants of connected unipotent groups}\label{7sunipotent}

Suppose that $X$ is an irreducible affine variety on which the
additive group $\Ga$ acts regularly. We have already seen that there
exists an algorithm that computes generators for a subalgebra
$R\subseteq S:=K[X]$ and generators of an ideal ${\mathfrak a}$ such
that $S^{\Ga}=(R:{\mathfrak a}^\infty)_S$. We now will deal with the
more general case where a connected unipotent group $N$ acts regularly
on $X$. A unipotent group $N$ is nilpotent (see
\mycite[Corollary~17.5]{Humphreys}) and therefore solvable. If
moreover $N$ is connected, then by [\citenumber{Humphreys},
Theorem~19.3] there exists a descending chain of normal subgroups
$$
N=N_k\supset N_{k-1}\supset N_{k-2}\supset \cdots\supset N_0=(0).
$$
such that each quotient $N_i/N_{i-1}$ has dimension one. By
[\citenumber{Humphreys}, Theorem~15.3(c)], each quotient is again
unipotent, and therefore it is isomorphic to the additive group $\Ga$
by [\citenumber{Humphreys}, Theorem~20.5].
This allows us to give a recursive approach to the computation of generators
of $K[X]^N$. The tricky part here is that $K[X]^{N_i}$ may not be finitely generated for some $i$, 
even if $K[X]^N$ is finitely generated.
\begin{alg} \label{7aunipotent}\ 

\begin{description}
\item{\bf Input:} The affine variety $X$ (given by its coordinate ring $S:=K[X]$), the connected unipotent group $N$
with its group structure (multiplication $N\times N\to N$ and inverse $N\to N$
and the identity element $e\in N$), the action $N\times X\to X$, and
a descending chain of normal subgroups
$$
N=N_{k}\supset N_{k-1}\supset \cdots\supset N_1\supset N_0=(0)
$$
with explicit isomorphisms $N_{i}/N_{i-1}\cong \Ga$ for $i=1,2,\dots,k$.
\item{\bf Output:} A subalgebra $T\subseteq K[X]$ and an ideal ${\mathfrak d}\subseteq T$
such that $K[X]^N=(T:{\mathfrak d}^\infty)_{K[X]}$.
 \end{description}
 \begin{enumerate}  
 \renewcommand{\theenumi}{\arabic{enumi}} 
 \item If $N=(0)$ (and $k=0$), then terminate with as output the algebra $S$ and its ideal $S$.
 \item Find a finitely generated subalgebra $R\subseteq S:=K[X]$ and an ideal ${\mathfrak a}$
such that $S^{N_1}=(R:{\mathfrak a}^\infty)_S$ as in Section~\ref{7sadditive}.
Say $R=K[f_1,\dots,f_r]$ and ${\mathfrak a}=(h_1,\dots,h_s)$.
\item Let $R'$ be the algebra generated by all $u\cdot f_i$ where $u\in N$ and $i=1,2,\dots,r$.
\item Let ${\mathfrak a}'$ be the ideal of $R'$ generated by all $u\cdot h_j$ where $u\in N$ and 
$j=1,2,\dots,s$.
\item Invoke this algorithm with input $R'$ and $N':=N/N_1$ to find a subalgebra 
$T\subseteq (R')^{N'}$
and an ideal ${\mathfrak c}$ of $T$ such that $(T:{\mathfrak c}^\infty)_{R'}=(R')^{N'}$.
\item Find a nonzero  element $a$ in $({\mathfrak a}')^{N'}={\mathfrak a}'\cap (R')^{N'}$.
Replace $T$ by $T[a]$ to ensure that ${\mathfrak a}'\cap T$ is not the zero ideal.
\item Output the algebra $T$ and the ideal ${\mathfrak d}:={\mathfrak c}({\mathfrak a'}\cap T)$.
\end{enumerate}
\end{alg}
Before we prove the correctness of this algorithm, we explain some of the steps in
more detail. 

In step (3), since $N_1$ is normal in $N$, $S^{N_1}$ is stable under $N$
and $R'\subseteq S^{N_1}$.

In step (6): Note that $N'$ is unipotent and ${\mathfrak a}'$ is nonzero.
We can find a nonzero finite dimensional subrepresentation $W\subseteq {\mathfrak a}'$
because $N'$ acts regularly on the infinite dimensional vector space ${\mathfrak a}'$.
But then $W^{N'}$ is nonzero. This shows that $({\mathfrak a}')^{N'}$ is nonzero.
A nonzero element in $({\mathfrak a}')^{N'}$ can be found using linear algebra
and exhaustive search.
\begin{proof}[Proof of correctness of Algorithm~\ref{7aunipotent}]
We need to show that
$$
S^N=(T:{\mathfrak d}^\infty)_S.
$$
We have 
$$
S^{N_1}=(R:{\mathfrak a}^\infty)_S.
$$
We claim that we also have
$$
S^{N_1}=(R':({\mathfrak a}')^\infty)_S.
$$
Suppose that $f\in S^{N_1}$.
Since $N$ is a normal subgroup, $S^{N_1}$ is $N$-stable. 
Let $W$ be the vector space spanned
by all $u\cdot f$, $u\in N$. Then $W$ is finite dimensional and contained
in $S^{N_1}=(R:{\mathfrak a}^\infty)_S$.
Then there exists a positive integer $l$ such that
$$
{\mathfrak a}^lW\subseteq R
$$
So in particular,
$$
{\mathfrak a}^l(u^{-1}\cdot f)\subseteq R.
$$
for all $u\in N$.
Applying $u$ gives
$$
(u\cdot {\mathfrak a})^l f\subseteq u\cdot R\subseteq R'.
$$
Since ${\mathfrak a}$ is finitely generated, there exists
finitely many elements $u_1,\dots,u_m$ such that
${\mathfrak a}'$ is generated by $u_i\cdot {\mathfrak a}$, $i=1,2,\dots,m$.

Since
$$
({\mathfrak a}')^{lm}=(u_1\cdot {\mathfrak a}+u_2\cdot {\mathfrak a}+\cdots
+u_{m}\cdot {\mathfrak a})^{lm}\subseteq u_1\cdot {\mathfrak a}^l+\cdots+u_m\cdot {\mathfrak a}^l
$$
we get
$$
({\mathfrak a}')^{lm}f\subseteq R'
$$
and 
$$
f\in (R':({\mathfrak a}')^{\infty})_S.
$$

Conversely, if $f\in (R':({\mathfrak a}')^{\infty})_S$, then $f$ is invariant under $N_1$
because $R'\subseteq S^{N_1}$ and ${\mathfrak a}'\subseteq S^{N_1}$ is not equal
to $(0)$.

Next we will show that
$$
S^N=(T:{\mathfrak d}^\infty)_S
$$
where
$$
{\mathfrak d}={\mathfrak c}({\mathfrak a}'\cap T).
$$

Suppose that $f\in S^N$. Then $f\in S^{N_1}=(R':({\mathfrak a}')^\infty)_S$, so there
exists a positive integer $l$ such that
$$
({\mathfrak a}')^l f\subseteq R'.
$$
It follows that
$$
({\mathfrak a}'\cap T)^l f\subseteq (R')^N=(R')^{N'}
$$
Since ${\mathfrak a}'\cap T$ is finitely generated, there exists a positive integer $m$
such that
$$
{\mathfrak c}^m({\mathfrak a}'\cap T)^l f\subseteq T.
$$
This shows that ${\mathfrak d}^nf\subseteq T$ for $n\geq \max\{l,m\}$ and
therefore $f\in (T:{\mathfrak d}^\infty)_S$.
It follows that
$$
S^N\subseteq (T:{\mathfrak d}^\infty)_S.
$$
The reverse inclusion
$$
S^N\supseteq (T:{\mathfrak d}^\infty)_S
$$
follows because $T,{\mathfrak d}\subseteq S^N$ and ${\mathfrak d}\neq (0)$.

\end{proof}
Finally we consider the case where $N$ is a connected unipotent group
acting regularly on an irreducible {\it factorial\/} variety $X$. In this case
we can effectively find a quasi-affine variety $U$ such that $K[X]^G=K[U]$.

\begin{alg}\label{7afactorial}\  

\begin{description}
\item{\bf Input:} The irreducible affine factorial variety $X$,
a connected unipotent group $N$ and a regular action $N\times X\to X$.

\item{\bf Output:} A finitely generated
subalgebra $\widetilde{R}\subseteq K[X]$ and an ideal ${\mathfrak g}\subseteq \widetilde{R}$
such that
$$K[X]^N=(\widetilde{R}:{\mathfrak g}^\infty)_{Q(\widetilde{R})}.$$
 \end{description}
 \begin{enumerate}  
 \renewcommand{\theenumi}{\arabic{enumi}} 
 \item Find a finitely generated subalgebra $R\subseteq K[X]$ and
 an ideal ${\mathfrak a}$ of $R$ such that 
 $$K[X]^N=(R:{\mathfrak a}^\infty)_{K[X]}$$
 using Algorithm~\ref{7aunipotent}. 
   \item Apply \aref{6afgli2} to find $\widetilde{R}$ and $\widetilde{g}$
   such that $(\widetilde{R}:\widetilde{\mathfrak g}^\infty)_{Q(\widetilde{R})}=K[X]^N$.
   \end{enumerate}
   \end{alg}
 \begin{proof}[Proof of correctness of Algorithm~\ref{7afactorial}]
 We need to show that \aref{6afgli2} applies here, i.e., 
 we have to prove that the variety corresponding to ${\mathfrak
   g}K[X]$ is equal to $K[X]$ or has codimension $\geq 2$.
 Suppose not. We can write $\sqrt{{\mathfrak g}K[X]}$ as the intersection
 of finitely many distinct prime ideals. One of these prime ideals has height $1$, say ${\mathfrak p}$ is
 such a prime ideal. Since $N$ is connected, ${\mathfrak p}$ must be stable under $N$.
 Since $K[X]$ is factorial, ${\mathfrak p}$ is a principal ideal, say ${\mathfrak p}=(h)$.
 Since $N$ is unipotent, it follows that $h$ is invariant under $N$, so $h\in K[X]^N$.
 
 We have already seen that $K[X]^N$ is isomorphic to the ring of regular functions
 on some quasi-affine variety $U$. There exists a finitely generated
 subalgebra $S$ of $K[X]^N$
 and an ideal ${\mathfrak b}$ of $S$ such that 
 $$
 K[X]^N=(S:{\mathfrak b}^\infty)_{Q(S)}.
 $$
 Clearly ${\mathfrak b}\subset {\mathfrak g}$ since $K[X]^N_f$ is finitely
 generated for all $f\in {\mathfrak b}$.
 Therefore ${\mathfrak b}\subseteq {\mathfrak g}K[X]\subseteq hK[X]$. 
 It follows that $h^{-1}{\mathfrak b}\subseteq K[X]$
 and $h^{-1}{\mathfrak b}\subseteq K[X]^N$. 
 This shows that $h^{-1}\in (S:{\mathfrak b}^\infty)_{Q(S)}=K[X]^N$.
 But $h^{-1}\not\in K[X]$, so $h^{-1}\not\in K[X]^N$. Contradiction.
 
 We have shown that the variety corresponding to ${\mathfrak g}K[X]$ has codimension $\geq 2$.
    \end{proof}
 \subsection{Invariants of arbitrary algebraic groups} \label{3sAlg}
 If $G$ is an arbitrary algebraic group, then there exists a connected unipotent
 normal subgroup $N$ such that $G/N$ is reductive. Suppose that $G$ acts on an
 irreducible affine variety $X$. One approach to compute generators of $K[X]^G$
 is by computing generators of $K[X]^N$ first. The problem of this is that $K[X]^N$ may
 not be finitely generated, even if $K[X]^G$ is finitely generated. 
 If $K[X]^N$ is finitely generated, then Algorithm~\ref{7aunipotent} can be used
 to compute a finitely generated subalgebra $R$ of $K[X]^N$ and an ideal ${\mathfrak a}$
 of $R$ such that $K[X]^N=(R:{\mathfrak a}^\infty)_{K[X]}$. Then Algorithm~\ref{alg4.8}
 can be used to find generators of $K[X]^N$. Finally Algorithm~\ref{2aInvarAff} can be
 used to compute generators of $K[X]^G=(K[X]^N)^{G/N}$ because $G/N$ is reductive.
 
 Even if $K[X]^N$ is not finitely generated, we might be able to compute generators
 of $K[X]^G$. Suppose that we have found $R$ and ${\mathfrak a}$
 such that $K[X]^N=(R:{\mathfrak a}^\infty)_{K[X]}$ using Algorithm~\ref{7aunipotent}. 
 Assume that $R=K[f_1,\dots,f_r]$ and ${\mathfrak a}=(h_1,\dots,h_s)$.
 We could try to copy the approach in Section~\ref{7sunipotent}. So
 let $R'$ be the algebra generated by $\sigma\cdot f_i$ with $\sigma\in G$ and $i=1,2,\dots,r$
 and let ${\mathfrak a}'$ be the ideal generated by all $\sigma\cdot h_j$ with
 $\sigma\in G$ and $j=1,2,\dots,s$. Similarly as in the proof of Algorithm~\ref{7aunipotent}
 we can show that
 $$
 K[X]^G = (R':({\mathfrak a}')^{\infty})_{K[X]}
 $$
 If $({\mathfrak a}')^{G/N}$ is not equal to the zero ideal, then one can show that
\begin{equation}\label{7eGinv}
 K[X]^G=((R')^{G/N}:(({\mathfrak a}')^{G/N})^{\infty})_{K[X]}.
 \end{equation}
 Generators of $(R')^{G/N}$ can be computed using Algorithm~\ref{2aInvarAff}.
 Generators of $({\mathfrak a}')^{G/N}={\mathfrak a}'\cap (R')^{G/N}$ can be computed by
 using the usual Gr\"obner basis techniques. 
 Finally generators of $K[X]^G$ can be found using (\ref{7eGinv}) and
 Algorithm~\ref{alg4.8}.  Of course this algorithm will not terminate, unless
 $K[X]^G$ is finitely generated.
 
 So what can we do if $({\mathfrak a}')^{G/N}$ is zero?
 Perhaps the choice of $R'$ and ${\mathfrak a}'$ were unfortunate. 
 Suppose that there exists an element $f\in K[X]^G$ such that $(K[X]^N)_f$
 is finitely generated. Then $f\in {\mathfrak g}$ where ${\mathfrak g}$
 is the finite generation locus ideal of $K[X]^N$. 
 Using Algorithm~\ref{6afgli} we can construct subalgebras
 $$
 \widetilde{R}_1\subseteq \widetilde{R}_2\subseteq \cdots
 $$
 and ideals
 $$
 {\mathfrak g}_1\subseteq {\mathfrak g}_2\subseteq \cdots
 $$
 such that $\bigcup \widetilde{R}_i=K[X]^N$ and ${\mathfrak g}=\bigcup_i{\mathfrak g}_i$.
 So we have $f\in {\mathfrak g}_i$ for some $i$.
 We terminate Algorithm~\ref{6afgli} at step $i$ when $f\in {\mathfrak g}_i$.
 We have
 $$
 K[X]^N=(R:{\mathfrak a}^\infty)_{K[X]}=(\widetilde{R}_i:{\mathfrak g}_i^\infty)_{K[X]}.
 $$
 So we might as well replace $R$ by $R=\widetilde{R}_i$ and ${\mathfrak a}$
 by ${\mathfrak a}={\mathfrak g}_i$.
 We  then still have
 $$
 K[X]^{N}=(R:{\mathfrak a}^\infty)_{K[X]}
 $$
 but we also have $f\in {\mathfrak a}^{G/N}$, so ${\mathfrak a}^{G/N}$ is not the zero ideal.
 We can proceed to compute generators of the invariant ring $K[X]^G$ as discussed before.
 
We just saw that there exists an algorithm to compute generators
of $K[X]^G$ if there exists a nonzero 
element $f\in K[X]^G$ such that $K[X]^N_f$ is finitely generated.
This may not always be the case as the following example shows.
\begin{example}
Let $H$ be the group and $X$ be the representation
in Nagata's counterexample to Hilbert's fourteenth problem (see \mycite{nag:a}).
Here $V$ is a $32$-dimensional representation and $H$ is an algebraic group over
the base field $K={\mathbb C}$ and $K[X]^H$ is not finitely generated.
Let $N$ be the unipotent radical of $H$. Then $N$ is a connected unipotent group,
$H/N$ is reductive, and $K[X]^N$ is not finitely generated, because
otherwise $K[X]^H=(K[X]^N)^{H/N}$ would be finitely generated.
Let ${\mathcal G}_m={\mathbb C}^\star$ be the multiplicative group acting by scalar multiplication,
and let $G={\mathcal G}_m N$. Then $N$ is the unipotent radical of $G$. 
Since $K[X]^G=K$, for every nonzero $f\in K[V]^G$ we have
$K[X]^N_f=K[X]^N$ is not finitely generated.
\end{example}


Suppose that $G$ is an algebraic  group and $X$ is an irreducible normal $G$-variety.
Suppose that the quotient field $Q(K[X]^G)$ of the invariant ring $K[X]^G$
is equal to the field of invariant rational functions on $X$, denoted by $K(X)^G$.
First we can find the transcendence degree of $K(X)^G:K$ as follows. 
Let $n=\dim X$ and let $m$ be the dimension of a generic $G$-orbit in $X$.
  Then the transcendence degree of $K(X):K(X)^G$ is $m$,
  and the transcendence degree of $K(X)^G:K$ is $n-m$.
  If we consider the morphism
  $$
  \psi:G\times X\to X\times X
  $$
  defined by
  $$
  \psi(\sigma,x)=(x,\sigma\cdot x)
  $$
  then the dimension of the closure of the image is $n+m$. Using Gr\"obner basis
  techniques one can compute the dimension of the Zariski closure
  of the image of $\psi$, and hence determine $m$.
  Using exhaustive search and linear algebra, one can compute 
  a linear basis of invariants $f_1,f_2,\dots\in K[X]^G$.
  Terminate this exhaustive search if one finds among these invariants
  $n-m$ algebraically independent functions. Let us call them $h_1,\dots,h_{n-m}$.
  Let $L$ be the field generated by $h_1,\dots,h_{n-m}$. Then $K(X)^G:L$ is an
  algebraic extension. Let $R$ be the integral closure of $K[h_1,\dots,h_{n-m}]$
  in $K[X]$. Generators of $R$ can be computed using Algorithm~\ref{3aIC}.
  We have $Q(R)=Q(K[X]^G)$. It follows that
  $$
  K[X]^G=K(X)^G\cap K[X]=Q(R)\cap K[X].
  $$
  So we can use Algorithm~\ref{5aHilb14} to find generators of $K[X]^G$.
  
\addcontentsline{toc}{section}{References}

\bibliographystyle{mybibstyle} \bibliography{bib}

\bigskip

\begin{center}
\begin{tabular}{ll}
  Harm Derksen & Gregor Kemper \\
  Department of Mathematics \hspace{10mm} & Technische Universit\"at
  M\"unchen \\
  University of Michigan & Zentrum Mathematik - M11 \\
  East Hall & Boltzmannstr. 3 \\
  530 Church Street & 85\,748 Garching \\
  Ann Arbor, MI 48109-1043 & Germany \\
  USA & \\
  {\tt hderksen$@$umich.edu} & {\tt
    kemper$@$ma.tum.de}
\end{tabular}
\end{center}

\end{document}